\documentclass[reqno,oneside,12pt]{amsart}
\usepackage{fancyhdr}
\usepackage{xspace}
\usepackage{geometry}
\usepackage{tikz}
\tikzset{node distance=1.5cm, auto}
\usepackage{color}
\definecolor{darkgreen}{rgb}{0,0.45,0}
\usepackage[pagebackref,colorlinks,citecolor=darkgreen,linkcolor=darkgreen]{hyperref}
\usepackage{comment}

\usepackage{graphicx}
\usepackage{enumerate}
\usepackage[all]{xy}
\usepackage{amsmath,amsfonts,amssymb}
\usepackage[latin9]{inputenc}

\newcommand{\Mod}{{\sf Mod}}

\newcommand{\Gal}{{\sf Gal}}
\newcommand{\coGal}{{\sf coGal}}
\newcommand{\biGal}{{\sf biGal}}
\newcommand{\BA}{{\sf BA}}
\newcommand{\BC}{{\sf BC}}

\newcommand{\Aut}{{\sf Aut}}

\newcommand{\YD}{{\sf YD}}

\def\lcan{\underline {\sf can}}
\def\rcan{\overline {\sf can}}

\newcommand{\Id}{\mathrm{id}}

\def\ot{\otimes}
\def\ul{\underline}
\def\ol{\overline}

\newcommand{\Bb}{\mathcal{B}}
\newcommand{\Cc}{\mathcal{C}}
\newcommand{\Dd}{\mathcal{D}}

\newcommand{\Gg}{\mathcal{G}}

\newcommand{\Mm}{\mathcal{M}}

\newcommand{\Vv}{\mathcal{V}}

\newcommand{\Zz}{\mathcal{Z}}

\def\tens{\otimes}
\def\boxtens{\boxtimes}
\def\cotens{\square}

\newcommand{\lbiprod}{{>\!\!\!\triangleleft\kern-.33em\cdot}}
\newcommand{\rbiprod}{{\cdot\kern-.33em\triangleright\!\!\!<}}

\newcommand{\la}{{\triangleright}}
\newcommand{\ra}{{\triangleleft}}
\def\bla{{\scalebox{.7}{\ensuremath \blacktriangleright}}}
\def\bra{{\scalebox{.7}{\ensuremath \blacktriangleleft}}}

\renewcommand{\o}{{}_{(1)}} 
\renewcommand{\t}{{}_{(2)}}
\renewcommand{\th}{{}_{(3)}}

\newcommand{\CC}{{\mathcal{C}}}
\newcommand{\CZ}{{\mathcal{Z}}}
\newcommand{\CM}{{\mathcal{M}}}

\newcommand{\can}{\mathsf{can}}

\newtheorem{prop}{Proposition}[section]
\newtheorem{proposition}[prop]{Proposition}
\newtheorem{lemma}[prop]{Lemma} 
\newtheorem{corollary}[prop]{Corollary} 
\newtheorem{theorem}[prop]{Theorem}

\theoremstyle{definition}

\newtheorem{definition}[prop]{Definition}
\newtheorem{example}[prop]{Example}
\newtheorem{examples}[prop]{Examples}
\newtheorem{remark}[prop]{Remark}
\newtheorem{remarks}[prop]{Remarks}

\newtheorem{notation}[prop]{Notation}

\newcommand{\kk}{k}

\newcommand{\benu}{\begin{enumerate}}
\newcommand{\enu}{\end{enumerate}}

\newcommand{\beqna}{\begin{eqnarray}}
\newcommand{\eqna}{\end{eqnarray}}
\newcommand{\beqnast}{\begin{eqnarray*}}
\newcommand{\eqnast}{\end{eqnarray*}}
\newcommand{\beqn}{\begin{equation}}
\newcommand{\eqn}{\end{equation}}
\newcommand{\beqnst}{\begin{equation*}}
\newcommand{\eqnst}{\end{equation*}}

\usepackage{latexsym}

\usepackage{xspace}
\usepackage{amscd}
\usepackage{amssymb}
\usepackage{amsfonts}

\makeatother

\newcommand{\bema}{\left ( \begin{array}}
\newcommand{\ema}{\end{array} \right )}

\newcommand{\Hom}{\operatorname{Hom}}
\newcommand{\End}{\operatorname{End}} 

\newif\ifstartedinmathmode
\newcommand\encircled[1]{%
  \relax\ifmmode\startedinmathmodetrue\else\startedinmathmodefalse\fi%
  \tikz[baseline,anchor=base]{%
  \node[draw,circle,outer sep=0pt,inner sep=0ex]
    {\ifstartedinmathmode$#1$\else#1\fi};}%
}

\begin{document}

\title[Yetter-Drinfeld modules, centers and groupoid-crossed bicategories]{Generalized Yetter-Drinfeld modules, the center of bi-actegories and groupoid-crossed braided bicategories}

\author[R. Aziz]{Ryan Aziz}
\email{ryan.kasyfil@sci.ui.ac.id}
\address{Departement of Mathematics, Universitas Indonesia, Indonesia}

\author[J. Vercruysse]{Joost Vercruysse}
\email{joost.vercruysse@ulb.be}
\address{D\'epartement de Math\'ematiques, Universit\'e Libre de Bruxelles, Belgium}

\date{\today}

\begin{abstract} 
We study the notion of the $E$-center $\mathcal{Z}_E(\mathcal{M})$ of a $(\mathcal{C}, \mathcal{D})$-biactegory (or bimodule category) $\mathcal{M}$, relative to an op-monoidal functor $E: \mathcal{C} \to \mathcal{D}$. Specializing this notion to the case $\mathcal{M} = {}_A\mathrm{Mod}$, $\mathcal{C}={}_H\mathrm{Mod}$, $\mathcal{D} = {}_K\mathrm{Mod}$, and $E \simeq C\otimes_H - : {}_H\mathrm{Mod} \to {}_K\mathrm{Mod}$, where $H$ and $K$ are bialgebras, $A$ is an $(H,K)$-bicomodule algebra and $C$ is a $(K,H)$-bimodule coalgebra, we show that this $E$-center is equivalent to the category of generalized Yetter-Drinfeld modules as introduced by Caenepeel, Militaru, and Zhu \cite{BookStef}. We introduce the notion of a double groupoid-crossed braided bicategory, generalizing Turaev's group-crossed braided monoidal categories, and show that generalized Yetter-Drinfeld modules can be organized in a double groupoid-crossed braided bicategory over the groupoids of Galois objects and co-objects.
\end{abstract}

\maketitle
\section{Introduction}

The construction \cite{Drinfeld} of the category of {\em Yetter Drinfeld modules} (or YD modules for short) over a Hopf algebra $H$ (say, with bijective antipode) provides an important tool for building braided monoidal categories. Those braided monoidal categories are in turn of central importance in for example Topological Quantum Field Theory, where they provide means to establish topological invariants. The fact that the category of YD modules is braided, rather than merely a monoidal category like the categories of modules and comodules over $H$, can be explained by the fact that it can be viewed as the {\em Drinfeld center} of any of the two latter categories, see e.g.\ \cite{ma:center}. On the other hand, when $H$ is finite dimensional, then the category of YD modules is moreover isomorphic to a category of modules over a suitably constructed Hopf algebra, called the {\em Drinfeld double} $D(H)$ of $H$. As its category of modules now is a braided monoidal category, $D(H)$ is not just a Hopf algebra but even a {\em quasi-triangular} Hopf algebra.

Over the years, several variations and generalizations of YD modules have surfaced. Notably, {\em anti-Yetter-Drinfeld modules} (AYD modules), as introduced in \cite{hajac}, appear naturally in the framework of Hopf cyclic cohomology. These modules, just as usual YD modules have both an action and coaction of the same Hopf algebra, however in their compatibility, the antipode is replaced by the inverse antipode. AYD modules do no longer form a monoidal category, but they are an {\em actegory}, also called module-category, over the monoidal category of usual YD modules, meaning that the tensor product of a YD module and an AYD module is again an AYD module. In order to describe the structure of tensor products of several AYD modules, the notion of {\em higher Yetter-Drinfeld} module was introduced in \cite{hassanzadeh}, for which other (odd) powers of the antipode appear in their compatibility conditions. These higher YD modules in turn were generalized in \cite{panaite} to so-called $(\alpha,\beta)$-Yetter-Drinfeld modules, replacing (even) powers of the antipode by arbitrary Hopf algebra automorphisms and where it was already observed that those modules are particular cases of the {\em generalized Yetter-Drinfeld modules} introduced earlier by Caenepeel, Militaru and Zhu in \cite{BookStef}. In the latter and most general setting, YD modules are simultaneous modules over an $(H,K)$-bicomodule algebra $A$ and comodules over a $(K,H)$-bimodule coalgebra satisfying a suitable compatibility, $H$ and $K$ being bialgebras.  

Already in \cite{BookStef} it was shown that the construction of the Drinfeld double can be generalized in such a way that (under finiteness conditions) generalized YD modules coincide with modules over a certain smash product build out of the bicomodule algebra $A$ and the dual of the bimodule coalgebra $C$. However, the relation between generalized YD modules and the center construction for monoidal categories, as well as the braided monoidal structure of generalized YD modules remained a mystery. In this paper we
solve the above problem by introducing the notion of an $E$-center $\Zz_E(\Mm)$ of a $(\Cc,\Dd)$-biactegory $\Mm$, where $\Cc$ and $\Dd$ are monoidal categories and $E: \Cc \to \Dd$ is an op-monoidal functor, see Definition \ref{def:centerbiact}. 
We prove in Theorem \ref{thm:mainthm} that the category of generalized YD modules is equivalent to the relative center $\CZ_E({}_A\Mod)$, where ${}_A\Mod$ is an $({}_H\Mod, {}_K\Mod)$-biactegory and the functor $E:{}_H\Mod\to {}_K\Mod$ is given by $C\ot_H-$. If furthermore $C$ is a bi-Galois co-object, then the category of generalized YD modules is an $E$-braided biactegory over the (classical) YD modules over $H$ and $K$, see Theorem~\ref{thm:generalbraid}.

We already explained that YD modules, providing constructions for braided monoidal categories, are as such linked to topological quantum field theories. In the last part of the paper, we explore further structure on the categories of generalized YD modules, which leads to a link with Turaev's Homotopy Quantum Field Theories. In \cite{panaite}, it was shown that $(\alpha,\beta)$-YD modules can be organized in a group-crossed braided monoidal category in the sense of Turaev. In Section \ref{se:bicats} we introduce {\em groupoid-crossed braided bicategories} as a $2$-dimensional version of Turaev's notion and show that generalized YD modules can be endowed with such a structure.

\subsection*{Acknowledgment}
This work was started while RA was working at the Universit\'e Libre de Bruxelles on a projected funded by European Union's Horizon 2020 research and innovation programme under the Marie Sklodowska-Curie grant agreement No 801505. 

JV thanks the FWB for support trough the ARC project ``From algebra to combinatorics, and back''.

\section{Yetter-Drinfeld Structures}\label{sec:YDstruc}

\subsection{Generalized Yetter-Drinfeld modules}\label{sec:genYD}
Throughout the paper, we work over a commutative ring $\kk$.
A (left,right) \textit{Yetter-Drinfeld datum} (or YD datum) is a four-tuple $(H,K,A,C)$ consisting of bialgebras $H$ and $K$, an $(H,K)$-bicomodule algebra $A$ with coactions
\[\Delta_L(a) =  a_{[-1]} \tens a_{[0]} \in H \tens A, \qquad \Delta_R(a) = a_{[0]} \tens a_{[1]} \in A\tens K,\] and a $(K,H)$-bimodule coalgebra $C$ with actions 
$\la : K\tens C \to C$ and $\ra : C \tens H \to C.$

\begin{definition}[{\cite[Section 4.4]{BookStef}}]
Let $(H, K, A, C)$ be a YD datum. A left-right generalized Yetter-Drinfeld module over a YD datum is a $\kk$-module $M$ such that $M$ is a left $A$-module with action $\cdot: A\tens M \to M$ and a right $C$-comodule with coaction $\rho^r: M\to M\tens C$ denoted by $\rho^{r}(m) = m_{[0]}\tens m_{[1]}$, and the following compatibility hold for all $a\in A$ and $m\in M$
\begin{equation}\label{eq:genYDcomp1}
(a_{[0]} m)_{[0]} \tens (a_{[0]} m)_{[1]}\ra a_{[-1]} = a_{[0]} m_{[0]} \tens a_{[1]}\la m_{[1]}.
\end{equation}
\end{definition}
In the case $H$ is a Hopf algebra with invertible antipode, \eqref{eq:genYDcomp1} is equivalent to
\begin{equation}\label{eq:genYDcomp2}
(am)_{[0]} \tens (am)_{[1]} = a_{[0]} m_{[0]} 
\tens a_{[1]} \la m_{[1]} \ra S^{-1}_H(a_{[-1]}).
\end{equation}
The category of left-right generalized YD modules over the YD datum $(H,K,A,C)$ is denoted by ${_A\YD^C}(H,K)$. As usual, generalized YD modules come with another three flavors: right-right, left-left, and right-left, and all of these flavors are equivalent by replacing the YD datum by their opposite or/and co-opposite counterparts.

\subsection{Usual Yetter-Drinfeld modules over a Hopf algebra}\label{se:classical}

Taking a trivial YD-datum $(H,H,H,H)$, where $H$ is a Hopf algebra that is endowed with its regular actions and coactions to make it in a bi(co)module (co)algebra over itself, we denote the associated category of YD modules simply by ${}_H\YD^H(H,H)={}_H\YD^H$. This is now a braided monoidal category with a strong monoidal forgetful functor to $k$-modules. More precisely, given two YD modules $M$ and $N$, we define a YD structure on the $k$-linear tensor product via the following action and coaction, defined for all $h\in H$ and $m\ot n\in M\ot N$:
\begin{eqnarray*}
H\ot M\ot N\to M\ot N, &\qquad & h\cdot (m\ot n) = h_{(1)}m\ot h_{(2)}n,\\
M\ot N\to M\ot N\ot H, && \rho(m\ot n) = m_{[0]}\ot n_{[0]}\ot n_{[1]}m_{[1]}.
\end{eqnarray*}
The braiding is given by 
$$\beta: M\ot N\to N\ot M, \qquad \beta(m\ot n)= n_{[0]}\ot n_{[1]}m.$$

\subsection{$(\alpha,\beta)$-Yetter-Drinfeld modules and (higher) anti-Yetter-Drinfeld modules}\label{sec:(alpha,beta)YD}
\label{se:alphabeta}
Let $H$ be a Hopf algebra and consider Hopf-algebra automorphisms $\alpha, \beta, \gamma, \delta \in \Aut_{\sf Hopf}(H)$.
We consider the YD datum $(H,H, A={}^\alpha H^\beta, C={}_\gamma H_\delta)$, where ${}^\alpha H^\beta$ and ${}_\gamma H_\delta$ have the same Hopf algebra structure as $H$, but with the modified coactions and actions as follows
\[h_{[-1]}\tens h_{[0]} = \alpha(h\o) \tens h\t, \qquad h_{[0]}\tens h_{[1]} = h\o \tens \beta(h\t), \qquad h\la g = \gamma(h) g, \qquad h\ra g = h \delta(g).\]
The compatibility condition on generalised Yetter-Drinfeld modules turns into
\begin{equation}\label{eq:(alpha,beta)-YDcomp1}
(h\t m)_{[0]} \tens (h\t m)_{[1]}\delta\alpha (h\o) = h\o m_{[0]} \tens \gamma\beta(h\t) m_{[1]}
\end{equation}
which is equivalent to
\begin{equation*}
(hm)_{[0]} \tens (hm)_{[1]} = h\t m_{[0]} 
\tens \gamma\beta(h\th) m_{[1]} \delta\alpha(S^{-1} h\o)
\end{equation*}
Notice that one could get the same compatibily condition by choosing $(A={}^{\gamma \beta}H^{\delta \alpha}, C=H)$ or $(A=H, C={}_{\alpha\delta} H_{\beta \gamma})$. In other words, the YD datum only depends on the products $\gamma \beta$ and $\delta \alpha$, which can be represented by a single pair of automorphisms $(\alpha',\beta')$.

Taking into account the above observations, and following \cite{panaite}, we call a generalized YD module over the YD datum $(H,H, {}^\alpha H^\beta, H)$ an {\em $(\alpha,\beta)$-Yetter-Drinfeld module}. Such a module is a left $H$-module right $H$-comodule satisfying the compatibility condition \eqref{eq:(alpha,beta)-YDcomp1} above specialized to the case $\gamma=\delta=id_H$. We denote the category of such YD-modules as 
${}_H\YD^H(\alpha,\beta)$.

The following results immediately follows from the above observations and direct computation.

\begin{lemma}\label{le:alphabeta}
Consider Hopf algebra automorphisms $\alpha,\beta,\gamma\in\Aut_{\sf Hopf}(H)$. Then we have the following isomorphisms of bicomodule algebras and bimodule coalgebras
\begin{eqnarray*}
\gamma: {^{\alpha\gamma} H^{\beta\gamma}}\to {^{\alpha}H^{\beta}}, &\quad& \gamma: {_\alpha H_\beta}\to {_{\gamma\alpha}H_{\gamma\beta}}
\end{eqnarray*}
Consequentlly, we have the following isomorphisms of categories.
\begin{eqnarray*}
{_H\YD^H}(\alpha,\beta)&=&{_H\YD^{{}_\beta H_\alpha}}(H, H)\cong {_H\YD^{{}_{\alpha^{-1}\beta} H}}(H, H)\\
&=& {_{{}^\alpha H^\beta}\YD^H}(H, H) \cong {_{H^{\beta\alpha^{-1}}}}\YD^H(H, H)
\end{eqnarray*}
\end{lemma}

\begin{remark}
The usual Yetter-Drinfeld modules are recovered when $\alpha = \beta = \Id$, while the anti-Yetter-Drinfeld modules as introduced in \cite{hajac} are recovered in the case of $\alpha=\Id$ and $\beta=S^2$. More generally, the case $\beta = S^{2i}$ is the notion of higher anti YD modules as studied in \cite{hassanzadeh}, in which case also  the anti-Yetter-Drinfeld contramodule \cite{brzezinski} are recovered for $i=-1$ if $M$ is finite-dimentional. The case $\beta=\Id$ with non-identity $\alpha$ is studied in \cite{canaepeel}.\end{remark}

\section{Center of biactegories}\label{sec:center}

\subsection{Preliminaries on monoidal categories and their actegories}
Recall (see e.g.\ \cite{maclane}) that a {\em monoidal} category is a category $\CC$ equipped with bifunctor $\tens : \CC\times \CC \to \CC$, unit object $1$, natural isomorphism $a_{U,V,W}: U\tens (V\tens W)\to (U\tens V) \tens W$ called associator, and natural isomorphisms $l_V: 1\tens V \to V $ and $r_V: V\tens 1 \to V$, such that all constraints obey suitable compatibility conditions. It is well-known by Mac Lane's coherence theorem that any monoidal category is monoidally equivalent to a strict monoidal category.

A {\em pre-braiding} on a monoidal category $\Cc$ is a natural transformation $\beta_{V,W}:V\ot W\to W\ot V$, natural in $V,W\in \Cc$, satisfying two hexagon conditions expressing the compatibility with the associators. In case $\beta_{V,W}$ are invertible, we say that $\Cc$ is {\em braided}, and in case the inverse of $\beta_{V,W}$ is given by $\beta_{W,V}$, we say that $\Cc$ is {\em symmetric}.

Given two monoidal categories $(\CC, \tens, 1_{\CC}, a, l, r)$ and $(\Dd, \boxtens, 1_{\Dd}, a', l', r')$, a functor $E: \CC \to \Dd$ is called an \textit{op-monoidal} functor if $E$ is equipped with morphisms $\psi_{0}: E1_{\CC}\to 1_{\Dd}$ and $\psi_{V,W}: E(V\tens W)\to EV\boxtimes EW$ natural in $V, W \in \CC$, such that
\begin{eqnarray*}
&El_V = l'_{EV}\circ (\psi_0\boxtens \Id_{EV})\circ \psi_{1_{\CC},V}, \qquad Er_V = r'_{EV}\circ (\Id_{EV}\boxtens \psi_0)\circ \psi_{V, 1_{\CC}}\\
&a'_{EV,EW,EZ}\circ(\Id_{EV}\boxtimes \psi_{W,Z})\circ\psi_{V,W\tens Z} = (\psi_{V,W}\boxtens \Id_{EZ})\circ(\psi_{V\tens W, Z})\circ E(a_{V,W,Z})
\end{eqnarray*}
for all $V,W,Z$ in $\Cc$.

An op-monoidal functor $E$ is called \textit{strong} if $\psi_0$ and $\psi_{V,W}$ are invertible for all $V,W\in \CC$. Dually, a functor $F: \CC \to \Dd$ is called a \textit{monoidal} functor if $F$ is equipped with morphisms $\gamma_0,\gamma$ where $\gamma_0 : 1_{\Dd} \to F(1_{\CC})$ and $\gamma_{V,W}: F(V)\boxtens F(W)\to F(V\tens W)$ natural in $V,W\in \CC$ such that they satisfy the arrow-reversing diagram of the condition for op-monoidal functor. Therefore, if an op-monoidal functor $(E,\psi_0, \psi)$ is strong, then it is also (strong) monoidal with $\gamma_0 = \psi^{-1}_0$ and $\gamma_{V,W} = \psi^{-1}_{V,W}$.

If, moreover, $(\CC, \beta)$ and $(\Dd, \beta')$ are braided monoidal categories, then an op-monoidal functor $E:\CC\to \Dd$ is called a \textit{braided op-monoidal} if $E$ satisfies the following additional condition for all $V,W\in \CC$
\[\psi_{W,V}\circ E(\beta_{V,W}) = \beta'_{E(W),E(V)}\circ \psi_{V,W}.\] Similarly, one can define a \textit{braided monoidal functor} $F:\CC\to \Dd$ by reversing the arrows in the diagram corresponding to the condition above.

For a given monoidal category $\CC$, we use the term \textit{left $\CC$-actegory} (sometimes called a left $\CC$-module category) for a category $\CM$ endowed with a bifunctor $\la : \CC \times \CM \to \CM$, and with isomorphisms
\[\lambda_M : 1_{\CC}\la M \to M, \qquad m_{V,W,M}: (V\tens W)\la M \to V\la (W\la M)\]
natural in $V,W\in \CC$ and $M\in \CM$, such that they satisfy suitable compatibility conditions with the monoidal structure on $\CC$. Similarly, a \textit{right $\CC$-actegory} is a category $\CM$ endowed with a bifunctor $\ra: \CM \times \CC \to \CM$, and with suitably compatible isomorphisms
\[\tau_M: M \ra 1_{\CC}\to M, \qquad  n_{M,V,W}: M\ra (V\tens W) \to (M\ra V)\ra W\]
natural in $V,W\in \CC$ and $M\in \CM$.

Let $\CM$ be a left $\CC$-actegory and a right $\Dd$-actegory. Then $\CM$ is said to be a \textit{$(\CC,\Dd)$-biactegory} if $\CM$ is furthermore equipped with isomorphisms
\[\eta_{V,M,X} : (V\la M) \ra X \to V\la (M\ra X)\]
natural in $V\in \CC$, $M\in \CM$, and $X\in \Dd$ subject to coherence conditions with respect to the monoidal structures of $\CC$ and $\Dd$. There is a version of Mac Lane's coherence theorem for (bi)actegories, and hence we will assume our (bi)actegories to be strict. A monoidal category $\CC$ is naturally a biactegory over itself where all bifunctors are given by the monoidal product of $\Cc$. Another natural example is given by the category ${}_A\Mod_B$ of $(A,B)$-bimodules over algebras $A$ and $B$, which is a $({}_A\Mod_A, {}_B\Mod_B)$-biactegory. Finally, if $H$ and $K$ are bialgebras and $A$ is an $(H,K)$-bicomodule algebra, then ${_A\Mod}$ is an $(_H\Mod,{}_K\Mod)$-biactegory by means of the $k$-linear tensor product.

\subsection{Center of biactegory relative to $E$}\label{sec:Ecenter}

Recall that for any monoidal category $\Cc$, we can construct a pre-braided monoidal category $\Zz^w(\Cc)$, called the {\em weak} (or {\em lax}) {\em center} of $\Cc$, as the category whose objects are pairs $(V,\beta^V)$, where $V$ is an object in $\CC$, and $\beta^V =\{\beta^V_U: U\tens V \to V\tens U\}_{U\in \CC}$ are natural transformations (called half-braidings) satisfying suitable compatibility conditions and such that the obvious forgetful functor $\Zz^w(\Cc)\to \Cc$ is strict monoidal. We call the subcategory $\CZ(\CC)$ of $\CZ^w(\CC)$ consisting of the pairs $(V,\beta^V)$ for which $\beta^V$ are natural isomorphisms the 
(strong) center of $\CC$.

\begin{definition}\label{def:centerbiact}
By a {\em center datum} we mean a quadruple $(\Cc,\Dd,\Mm,E)$, where $\Cc, \Dd$ are monoidal categories, $\Mm$ is a $(\Cc,\Dd)$-biactegory and $E: \Cc \to \Dd$ be an op-monoidal functor with structure morphisms $(\psi, \psi_0)$. Let us fix such a datum now.

The {\em lax $E$-center} of $\Mm$, denoted by $\Zz^w_E(\Mm)$, is the category defined as follows:
\begin{itemize}
\item Objects are pairs $(M, \beta^M)$, where $M\in \Mm$, and $\beta^M$ is a natural transformation 
\[\beta^M_V : V\la M \to M\ra E(V)\]
natural in $V\in \Cc$, such that the following diagram commutes for all $V,W \in \CC$ (heptagon condition):
\[\xymatrixcolsep{5pc}\xymatrix{
& V\la (W\la M) \ar[dr]^-{\Id_V \tens \beta^M_W} \ar[dl]_-\cong \\
(V\tens W)\la M  \ar[d]_-{\beta^M_{V\tens W}} && V\la (M\ra E(W)) \ar[dd]^-{\cong}\\
M\ra E(V\tens W) \ar[d]_-{\Id_M \tens \psi_{V,W}}  && \\
M\ra (E(V)\tens E(W)) \ar[dr]_-{\cong}
 && (V\la M) \ra E(W) \ar[dl]^-{\beta^M_V \tens \Id_{E(W)}}\\
&(M \ra E(V)) \ra E(W) &
}\]
which, if we disregard the associativity, can be written concretely as
\begin{equation}\label{eq:halfbraid}
(\Id_M \tens \psi_{V,W})\circ \beta^M_{V\tens W} = (\beta^M_V \tens \Id_{E(W)})\circ (\Id_V \tens \beta^M_W).
\end{equation}

\item A morphism $f: (M, \beta^M)\to (N, \beta^N)$ in $\Zz^w_E(\Mm)$ is a morphism $f: M\to N$ in $\Mm$ such that for all $V\in \CC$, the following diagram is commutative
\[\xymatrixcolsep{5pc}\xymatrix{
V\la M \ar[r]^-{\beta^M_V} \ar[d]_-{\Id_V \la f} & M\ra E(V) \ar[d]^-{f\ra \Id_{E(V)}}\\
V\la N \ar[r] \ar[r]^-{\beta^N_V} & N\ra E(V)
}.\] 
\end{itemize}

The strong $E$-center $\Zz_E(\Mm)$ is the full subcategory of the lax center consisting of those objects $(M, \beta^M)$ for which $\beta^M$ is invertible.
\end{definition}

\begin{remarks}\label{dualcenter}
\begin{enumerate}
\item In case $E$ is monoidal rather then op-monoidal, it is still possible to define a weak $E$-center of a $(\Cc,\Dd)$-biactegory, reversing the arrow induced by $\psi$ in the above heptagon condition.

\item
There are obvious other sided and dual notions of center data, switching to opposite categories and reversed monoidal structures.
More precisely, we call $(\Cc,\Dd,\Mm,E)$ a {\em dual center datum} if and only if $(\Dd^{op,rev},\Cc^{op,rev},\Mm^{op,rev},E)$ is a center datum, where ``op'' refers to the opposite category and ``rev'' to the reversed monoidal structure. Explicitly, this means that $\Mm$ is a $(\Cc,\Dd)$-biactegory and $E:\Dd\to\Cc$ is a monoidal functor. For such a datum we can again define the weak and strong center, in which case a half braid will be of the form
$$\beta^M_V:E(X)\la M\to M\ra X,$$
for $M\in \Mm$ and $X\in\Dd$.

If we only take the reversed monoidal structures and not opposite categories, one could obtain braidings of the form $\beta^M_V:M\ra X\to E(X)\la M$ and similarly only using opposite categories but no reversed monoidal structure braiding would be of the type $\beta^M_V:M\ra E(V)\to V\la M$. 
\item
Taking $\Cc=\Dd$ and $E$ to be an identity functor, we recover the center of ${}_{\Cc}\Mm_{\Cc}$ \cite{gelaki}. Additionally, if $\Mm = \Cc$, then $\Zz({}_{\Cc}\Cc_{\Cc}) = \Zz(\Cc)$ is the Drinfeld center for monoidal category \cite{ma:center} as recalled above. As a partial converse of the above, in case $E$ is strong op-monoidal, then $\Mm$ can be viewed as a $(\Cc,\Cc)$-biactegory, defining $M\ra V := M\ra E(V)$ for all $M\in \Mm$ and $V\in \Cc$. In this case our definition reduces to the center defined in \cite{gelaki}. 
\item 
Given a braided monoidal category $\Cc$ and left $\Cc$-actegory $\Mm$, the {\em reflexive center} of $\Mm$ with respect to $\Cc$ was introduced in \cite{Laugwitz}. Since $\Mm$ is a left $\Cc$-actegory, it is also a right $\Cc^{rev}$-actegory, where $\Cc^{rev}$ denotes the category $\Cc$ but with reversed tensor product as monoidal structure. Since $\Cc$ is braided, there is a monoidal functor $E:\Cc\to\Cc^{rev}$, which is just the identity functor on $\Cc$, but endowed with a monoidal structure
$$\beta_{V,V'}:V\ot V'\to V\ot^{rev} V'=V'\ot V$$
given by the braiding $\beta$ on $\Cc$. Then the $E$-relative center of the $(\Cc,\Cc^{rev})$-biactegory $\Mm$ coincides exactly with the reflexive center of $\Mm$.
\end{enumerate}
\end{remarks}

Let us now show that the biactegory structure of $\Mm$ lifts to its $E$-center. 

\begin{lemma}\label{biactlift}
Let $\Cc, \Dd$ be monoidal categories equipped with an op-monoidal functor $E:\CC\to \Dd$, and let $\Mm$ be a $(\Cc, \Dd)$-biactegory. Let $\CZ^w(\CC)$ and $\CZ^w(\Dd)$ be the weak center of $\CC$ and $\Dd$ respectively. 
Then the weak $E$-center $\Zz^w_E(\Mm)$ is a $(\Zz^w(\Cc), \Zz^w(\Dd))$-biactegory such that the following diagrams of functors commute
\[
\xymatrix{
\Zz^w(\Cc) \times \Zz^w_E(\Mm) \ar[d] \ar[rr]^{\la} && \Zz^w_E(\Mm) \ar[d]\\
\Cc \times \Mm \ar[rr]^{\la} && \Mm
} \qquad
\xymatrix{
\Zz^w_E(\Mm) \times \Zz^w(\Dd) \ar[rr]^{\ra} \ar[d] && \Zz^w_E(\Mm) \ar[d]\\
\Mm\times \Dd \ar[rr]^{\ra} && \Mm
}
\]
The vertical arrows in this diagram are the obvious forgetful functors. The commutativity of this diagram expresses that the biactegory structure on $\Zz^w_E(\Mm)$ is lifted from $\Mm$.

Explicitly, for all $(V, \beta^V) \in \CZ^w(\CC)$, $(X, \beta^X)\in \CZ^w(\Dd)$, and $(M, \beta^M)\in \CZ^w_E(\CM)$, the left and right action on $\CZ^w(\Mm)$ are given by
\[(V, \beta^V)\la (M, \beta^M)  = (V\la M, \beta^{V\la M}), \qquad (M, \beta^M)\ra (X, \beta^X) = (M \ra W, \beta^{M\ra X})\]
where
\[\beta^{V\la M}_W = (\Id_{EV} \la \beta^M_W)(\beta^V_W \la \Id_M), \qquad \beta^{M\ra X}_W = (\Id_M \ra \beta^X_{EW})(\beta^M_W \ra \Id_X)\]
are natural in $W\in \CC$. 

The same results hold if we replace weak by strong centers in the above.
\end{lemma}

\begin{proof}
Let $(M,\beta^M) \in \CZ^w_E(\CM)$, and $(X,\beta^X) \in \CZ^w(\Dd)$. One can see that $(M\ra X, \beta^{M\ra X})$ is indeed in $\CZ^w_E(\CM)$ since $\beta^{M\ra X}$ obeys the half-braiding axiom which is shown by the commutativity of the following diagram for any $U,V\in \Cc$. Basing on coherence for biactegories, we will not write associativity constraints in the diagram and denote all of $\ra$, $\la$ and $\ot$ by $\cdot$.
\[\xymatrix{
U\cdot V\cdot M \cdot X \ar[rr]^-{\Id\cdot \beta^M_V \cdot \Id} \ar[d]_-{\beta^M_{U\cdot V}\cdot \Id} 
&& U\cdot M \cdot EV\cdot X \ar[rr]^-{\Id\cdot \beta^X_{EV}} \ar[d]^-{\beta^M_U\cdot \Id}
&& U\cdot M\cdot X \cdot EV  \ar[d]^-{\beta^M_U\cdot \Id}\\
M\cdot E(U\cdot V)\cdot X \ar[d]_-{\Id\cdot \beta^X_{E(U\cdot V)}} \ar[rr]^-{\Id \cdot \psi_{U,V}\cdot \Id} && 
M\cdot EU\cdot EV\cdot X \ar[rr]^-{\Id\cdot \beta^X_{EV}} \ar[drr]_-{\Id\cdot \beta^X_{EU\cdot EV}}
&& M\cdot EU \cdot X \cdot EV \ar[d]^-{\Id\cdot \beta^X_{EU}\cdot \Id} \\
M\cdot X\cdot E(U\cdot V) \ar[rrrr]_-{\Id\cdot \psi_{U,V}}&& && M\cdot X \cdot EU \cdot EV 
}
\]
In the above diagram, the commutativity of the upper left square is exactly the heptagon condition for the half-braid $\beta^M$. The upper right square follows from the functoriality of the monoidal action $\ra$ (represented by the central $\cdot$ in the diagram). Lower quadrangle commutes by the naturality of $\beta^X$, and the triangle commutes as this is the hexagon condition for the half-braid $\beta^X$.

Next, we check that $\CZ^w_E(\CM)$ is a right $\CZ^w(\Dd)$-actegory. Since $\CM$ is a right $\Dd$-actegory, we have a natural isomorphism $n_{M,X,Y}: M\ra (X\tens Y) \to (M \ra X) \ra Y$. Therefore, we need to check that $n_{M,V,W}$ is a morphism in $\CZ^w_E(\CM)$. Again, since it is sufficient to check this in the strict case, we may suppose that $n_{M,X,Y}$ is the identity morphism in $\CM$, and we just need to check that the two half-braids on this object coincide. This follows from the commutativity of the following diagram.
\[ \xymatrixcolsep{5pc}\xymatrix{
U \cdot M\cdot X \cdot Y \ar@{=}[r] \ar[d]_-{\beta^M_U} & U \cdot M\cdot X\cdot Y \ar[d]^-{\beta^M_U}\\
M \cdot EU \cdot X \cdot Y \ar@{=}[r] \ar[d]_-{\Id\cdot \beta^{X}_{EU}\cdot \Id}  & M \cdot EU \cdot X\cdot Y \ar[dd]^-{\beta^{X\cdot Y}_{E(U)}}\\
M\cdot X \cdot EU \cdot Y \ar[d]_-{\Id \cdot \beta^Y_{EU}} & \\
M \cdot X \cdot Y \cdot EU \ar@{=}[r] & M \cdot X\cdot Y \cdot EU.
}\]
Here the lower rectangle commutes by the definition of the monoidal structure on $\Cc^w(\Dd)$.
One can also check that $\tau_{M}: M\ra 1_\Dd \to M$ is a morphisms in $\CZ^w_E(\CM)$ using similar argument. 
We do not need to check the axioms for $\tau$ and $n$ as these already follow from the fact that they are the same coherence morphisms as those of the $\Dd$-actegory structure on $\Mm$.
Using similar argument, one can show that $\CZ^w(\Cc)$ acts on $\CZ^w_E(\CM)$, making $\CZ^w_E(\CM)$ a $(\CZ^w(\Cc), \CZ^w(\Dd))$-biactegory.
\end{proof}

\subsection{$E$-braided biactegories}

Given a monoidal category $\CC$, then by the construction, its (lax) center $\CZ^{(w)}(\CC)$ is a (pre-)braided category. Our next aim is to show a similar result for the $E$-center of biactegory. To this end we now define the notion of $E$-braiding for a biactegory.

\begin{definition}\label{def:E-prebraided}
Let $\CC$ and $\Dd$ be prebraided categories (denoting -by abuse of notation- both braidings by $\beta$), and let $E: \CC \to \Dd$ be op-monoidal. A $(\Cc,\Dd)$-biactegory $\CM$ is called \textit{$E$-prebraided} if there exists morphisms
\[\beta_{V,M}: V\la M \to M\ra E(V),\]
natural in $V\in \Cc$ and $M\in \Mm$,
such that for all $M\in \CM$, $V,W\in \Cc$, and $X\in \Dd$, the following diagrams are commutative
\begin{equation}\label{prebraidedact1}
{\xymatrix{
{} & V\la (W\la M)\ar[ld]_-{\Id_V \la \beta_{W,M}} \ar[rd]^-{\cong} &{}\\
V\la(M\ra E(W)) \ar[dd]_-{\cong} & {} & (V\tens W)\la M \ar[d]^-{\beta_{V\tens W,M}}\\ 
&& M \ra E(V\tens W) \ar[d]^-{\Id_M \tens \psi_{V,W}}\\
(V\la M)\ra E(W)\ar[rd]_-{\beta_{V,M} \ra \Id_{E(W)}} & {} & M\la (E(V)\tens E(W))\ar[ld]^-{\cong}\\
{} & (M\ra E(V))\ra E(W) &{}
}}
\end{equation}
\begin{equation}\label{prebraidedact2}
\hspace{-.5cm}\resizebox{8.7cm}{!}{
{\xymatrix{
{} & (V\tens W) \la M \ar[ld]_-{\beta_{V,W} \la \Id_M} \ar[rd]^-{\cong} & {}\\
(W\tens V)\la M \ar[d]_-{\cong} & {} & V\la (W\la M) \ar[d]^-{\beta_{V,W\la M}}\\
W\la (V\la M) \ar[rd]_-{\Id_W \la \beta_{V,M}} & {} & (W\la M)\ra E(V) \ar[ld]^-{\cong}\\
{} & W\la (M\ra E(V)) & {}
}}}
\quad
\resizebox{8.7cm}{!}{
{\xymatrix{
{}& (V\la M) \ra X\ar[ld]_-{\beta_{V,M} \ra \Id_X} \ar[rd]^-{\cong} &{}\\
(M\ra E(V)) \ra X \ar[d]_-{\cong} & {} & V\la (M\ra X)\ar[d]^-{\beta_{V,M\ra X}}\\
M\ra(E(V)\tens X) \ar[rd]_-{\Id_M\ra \beta_{E(V), X}} & {} & (M\ra X)\ra E(V) \ar[ld]^-{\cong}\\
{} & M \ra (X\tens E(V)) &{} 
}}}
\end{equation}
where $\beta_{E(V),X}$ is a braiding in $\Dd$ and $\beta_{V,W}$ is a braiding in $\Cc$. Moreover, $\CM$ is called an \textit{$E$-braided category} if $\beta_{V,M}$ is a natural isomorphism for all $V\in \CC$ and $M\in \CM$.
\end{definition}

In case $\CM=\CC=\Dd$ and $E$ is an identity functor, we see that the two diagrams from \eqref{prebraidedact2} coincide and we recover the usual hexagon conditions for a braided monoidal category. 
In the next proposition, we see that the weak $E$-center of $(\Cc,\Dd)$-biactegory $\CM$ forms an $E$-prebraided category.

\begin{theorem}\label{thm:centerbiact} 
Let $\Cc, \Dd$ be monoidal categories such that there is commutative diagram of functors  
\[
\xymatrix{
\Zz^w(\Cc) \ar[d] \ar[rr]^E && \Zz^w(\Dd) \ar[d]\\
\Cc \ar[rr]^E && \Dd
}
\]
where the vertical arrows are the forgetful functors, the upper horizontal functor is braided op-monoidal and the lower horizontal functor is op-monoidal. The commutativity of diagram expresses that the lower op-monoidal functor $E$ lifts to a braided op-monoidal functor between the weak centers, by abuse of notation we denote both horizontal op-monoidal functors by $E$.

Furthermore, let $\Mm$ be a $(\Cc, \Dd)$-biactegory and consider $\Zz^w_E(\Cc)$ as $(\Zz^w(\Cc),\Zz^w(\Dd))$-biactegory as in Lemma~\ref{biactlift}. Then $\Zz^w_E(\Cc)$ is $E$-prebraided, with braiding
\[\beta_{V,M}:= \beta^M_V : (V\la M,\beta^{V\la M}) \to (M\ra EV,\beta^{M\ra EV})\]
for all $(V,\beta^V) \in \CZ^w(\CC)$ and $(M,\beta^M)\in \CZ^w_E(\CM)$.

Similarly, the strong $E$-center $\Zz_E(\Cc)$ is an $E$-braided $(\Zz(\Cc),\Zz(\Dd))$-biactegory.
\end{theorem}
\begin{proof}
Let us check the $E$-prebraiding $\beta_{V,M}$ as stated above is a morphism in $\CZ^w_E(\CM)$. This follows by commutativity of the following diagram, where we again work in the strict monoidal case and denote all of $\la$, $\ra$ and $\ot$ by $\cdot$.
\[\xymatrix{
W\cdot V\cdot M \ar[r]^-{\beta_{V,W}\cdot \Id} \ar[dd]_{\Id \cdot \beta_{V,M}} \ar[rd]_-{\beta_{W\cdot V,M}} & V\cdot W\cdot M \ar[rr]^-{\Id\cdot \beta_{W,M}} \ar[rd]^-{\beta_{V\cdot W,M}} && V \cdot M \cdot EW \ar[dd]^-{\beta_{V,M}\cdot\Id}\\
& M\cdot E(W\cdot V) \ar[dr]_-{\Id\cdot \psi_{W,V}} \ar[r]^-{\Id\cdot E\beta_{W,V}} & M\cdot E(V\cdot W)\ar[dr]^-{\Id\cdot \psi_{V,W}} & \\
W \cdot M\cdot EV \ar[rr]_-{\beta_{M,V}\cdot \Id}&& M \cdot EW \cdot EV\ar[r]_-{\Id\cdot\beta_{EW,EV}} & M \cdot EV \cdot EW.
}\]
The left lower and upper right triangles are commutative due to the half-braid condition of $\beta^M$. The commutativity of the upper parallelogram is due to the naturality of $\beta^M$, the  lower parallelogram follows from the fact that $E$ is braided op-monoidal.
Since $\beta^V_M$ is a morphism in $\CZ^w_E(\CM)$, one can show that the three diagrams in Definition \ref{def:E-prebraided} are satisfied using that $\CZ^w(\CC)$ and $\CZ^w(\Dd)$ are pre-braided categories. 
\end{proof}

We finish this section with an obvious sufficient condition for a lifting as needed in Theorem~\ref{thm:centerbiact}. In fact, we formulate this lifting in a slightly generalized setting.

\begin{lemma}\label{liftingtocenter}
Let $E:\Cc\to\Dd$ be a monoidal equivalence, whose quasi-inverse we denote by $\ol E$ and $F:\Cc\to\Cc$ an op-monoidal functor. Then $E$ lifts to an equivalence between relative centers as in the diagram below, where the vertical arrows indicate the obvious forgetful functors
\[
\xymatrix{
\Zz^w_F(\Cc) \ar[d] \ar[rr]^-E && \Zz^w_{EF\ol E}(\Dd) \ar[d]\\
\Cc \ar[rr]^-E && \Dd
}
\]
\end{lemma}

\begin{proof}
Since $E$ is a monoidal equivalence, the lifting of the functor $E$ is just an instance of {\em transfer of structure} along this equivalence. Explicitly, let $(X,\beta^X)$ be an element of $\Zz^w_F(\Cc)$, where $\beta^X_Y:Y\ot X\to X\ot FY$ is the half-braid. Then we define a half braid $\beta^{EX}_Z: Z\ot EX\to EX\ot EF\ol EZ$ for any $Z\in \Dd$ by means of the following composition.
\begin{eqnarray*}
Z\ot EX &\cong & E\ol EZ \ot EX \cong E(\ol EZ\ot X)\\
&\stackrel{E\beta^X_{\ol EZ}}{\longrightarrow}& 
E(X\ot F\ol EZ) \cong EX\ot EF\ol EZ
\end{eqnarray*}
The above isomorphisms are obtained from the monoidal equivalence $E$. We leave it to the reader to check that $\beta^{EX}$ is indeed a half-braid, that this construction is functorial and leads to an equivalence of categories.
\end{proof}

\section{The relative center of a category of modules}\label{sec:centerYD}

In this section we will apply the general categorical notions and results from the previous section to the case where we consider module categories\footnote{By ``module categories'' we mean ``categories of modules''. Recall that we use the term {\em actegory} for a category endowed with an action of a monoidal category, which is also sometimes called ``module category'' in literature.} over $k$-algebras, and where all monoidal structures are induced by the $k$-linear tensor product. 

Recall that the category ${}_H\Mod$ of left modules over a $k$-algebra $H$ is monoidal with respect to the $k$-linear tensor product if and only if $H$ has a bialgebra structure.
In this case, it is well-known (see, e.g.,\ \cite{ma:book}) that the weak center $\CZ^w({}_H\Mod)$ is braided monoidally isomorphic to the category of left-right Yetter-Drinfeld modules ${}_H\YD^H$. 
If moreover $H$ is a Hopf algebra, then any object from the weak center is also in the strong center, so that we find $\CZ^w({}_H\Mod)\cong \CZ({}_H\Mod) \cong {}_H\YD^H$ in this case. Under the condition that the base ring $k$ is a field, by a dual reasoning, we can see that also $\CZ^w(\Mod^H)$ is braided monoidally isomorphic to ${}_H\YD^H$ for any bialgebra $H$, and further more coincides with $\CZ(\Mod^H)$ if $H$ is Hopf. 

\subsection{The category of generalized Yetter-Drinfeld modules as relative center}

Let us start by describing a center datum in the setting of categories of modules. In view of what we recalled above, considering two bialgebras $H$ and $K$ is equivalent to considering two monoidal categories $\Cc={}_H\Mod$ and $\Dd={}_K\Mod$, whose forgetful functors to $k$-modules are strict monoidal. Furthermore, for another $k$-algebra $A$, we find that $A$ has an $(H,K)$-bicomodule algebra structure if and only if $\Mm={_A\Mod}$ is an $({}_H\Mod,{}_K\Mod)$-biactegory where the monoidal actions $\la, \ra$ are given by the $k$-linear tensor product. Under this correspondence, for $M\in {}_A\Mod$, $V\in {}_H\Mod$ and $W\in {}_K\Mod$, the $A$-action  on $V\ot M\ot W$ is given by the formula
\[a\cdot(v\ot m\ot w)=a_{[-1]}\la v\ot a_{[0]}m\ot a_{[1]}\ra w,\]
for all $m\in M, v\in V, w\in W$, and $a\in A$.

By the classical Eilenberg-Watts theorem, any functor co-continuous functor $E:\Cc\to\Dd$ is of the form $E\simeq C\ot_H-:{_H\Mod}\to {_K\Mod}$ for some $(K,H)$-bimodule $C$. Furthermore, we find (for details, see e.g.\ \cite{AV})
\begin{lemma}
For a co-continuous functor $E\simeq C\ot_H-:{_H\Mod}\to {_K\Mod}$, there is a bijective correspondence between:  
\begin{enumerate}[(1)]
\item op-monoidal structures on $E$;
\item monoidal structures on the right adjoint of $E$;
\item coalgebra structures on $C$ turning it into an $(K,H)$-bimodule coalgebra.
\end{enumerate}
\end{lemma}

Under the above correspondence, the op-monoidal structure on the functor $C\ot_H-$ is given by
\begin{eqnarray}
&&\psi_{V,W}:C\ot_H(V\ot W)\to C\ot_H V\ot C\ot_H W,\nonumber 
\\&&\hspace{3cm}\psi_{V,W}(c\ot_H (v\ot w))=c_{(1)}\ot_H v\ot c_{(2)}\ot_H w; \label{Eopmonoidal}\\
&&\psi_0:C\ot_H k\to k,\ \psi_0(c\ot_H \lambda)=\epsilon(c)\lambda
\end{eqnarray}
for all $v\in V$, $w\in W$, $c\in C$, $\lambda\in k$, where $V$ and $W$ are arbitrary left $H$-modules.

The above observations are summarized in the following proposition.

\begin{proposition}\label{correspondencedata}
There is a bijective correspondence between:
\begin{enumerate}[(1)]
\item Center data $(\Cc,\Dd,\Mm,E)$, where $\Cc$, $\Dd$ and $\Mm$ are module categories over $k$-algebras and the forgetful functors send the $(\Cc,\Dd)$-biactegory $\Mm$ to the $(\Mod_k,\Mod_k)$-biactegory $\Mod_k$, and where $E$ is co-continuous;
\item Yetter-Drinfeld data $(H,K,A,C)$.
\end{enumerate}
\end{proposition}

We are now ready to state our main result of this section.

\begin{theorem}\label{thm:mainthm}
Let $(H,K,A,C)$ be a Yetter-Drinfeld datum and $({}_H\Mod,{}_K\Mod,{}_A\Mod,E\simeq C\ot_H-)$ the corresponding center datum via Proposition~\ref{correspondencedata}. Then the lax $E$-center of the $({}_H\Mod, {}_K\Mod)$-biactegory ${}_A\Mod$ of left $A$-modules is isomorphic to the category of generalized YD modules:
$$\CZ^w_{C\ot_H-}({}_A\Mod)\cong {}_A\YD^C(H,K).$$
\end{theorem}

\begin{proof}
An object in $\CZ^w_{C\tens_H -}({}_A\Mod)$ consist of an $A$-module $M$ together with a half-braid $\beta^M_V: V\ot M\to M\ot (C\ot_H V)$ natural in $V \in {}_H\Mod$. As usual, by naturality $\beta^M$ is completely determined by its component in the regular $H$-module:
\[\beta^M_H: H\ot M\to M\ot (C\ot_H H)\cong M\ot C.\]
We write $\beta^M_H(h\tens m) = m_\beta \tens h^\beta$.

Define a map $\lambda : A\tens M\to H\tens M$ by $\lambda(a\tens m) = a_{[-1]}\tens a_{[0]}m$. Using $\lambda$, we are able to define the following map
\[\xymatrixcolsep{3pc}\xymatrix{\Phi: A\tens M \ar[r]^-{\lambda} & H\tens M \ar[r]^-{\beta^M_H} & M\tens C
}.\]
One can check $\Phi$ is left $A$-linear by using the free action of $A$ on $A\tens M$ and the diagonal action on $M\tens C$. Thus, we have ${}_A\Hom(A\tens M, M\tens C)\cong \Hom(M, M\tens C)$, which allows us to define $\rho^r : M\to M\tens C$ by
\begin{equation}\label{eq:rho}\rho^r(m) = \Phi(1_A\tens m) = m_\beta \tens 1_H{}^\beta.\end{equation}

We first check that $\rho^r$ is coassociative. Ignoring associativity of the monoidal structures, consider the following diagram
\[\xymatrixcolsep{5pc}\xymatrix{
H\tens M \ar[r]^-{\Delta_H \tens \Id_M} \ar[d]_-{\beta^M_H} & H\tens H \tens M \ar[r]^-{\Id_H \tens \beta^M_H} \ar[d]^-{\beta^M_{H\tens H}} & H\tens M\tens C\ot_H H \ar[d]^-{\beta^M_H \tens \Id_C}\\
M\tens (C\tens_H H) \ar[r]^-{\Id_M \tens (\Id_C \tens_H \Delta_H)} \ar[d]_-{\cong} & M\tens (C\tens_H (H\tens H)) \ar[r]^-{\Id_M \tens \psi_{H,H}} & M\tens C\ot_H H\tens C\ot_H H \ar[d]^\cong\\
M\tens C \ar[rr]^-{\Id_M \tens \Delta_C} && M\ot C\ot C
}.\]
The left-square is commutative by naturality of $\beta^M$ in ${}_H\Mod$, the right-square is commutative by the condition of $\beta^M$ being a half-braiding in ${}_A\Mod$, and the bottom-triangle is commutative since $C$ is a right $H$-module coalgebra and the op-monoidal structure on $E$ is given by \eqref{Eopmonoidal}.
Therefore, we obtained
\[(\Id_M \tens \Delta_C)\circ \beta^M_H = (\beta^M_H \tens \Id_C)\circ (\Id_H \tens \beta^M_H)\circ(\Delta_H\tens \Id_M) \]
which translates via \eqref{eq:rho} to the coassociativity of $\rho^r$.

Now we show that $\rho^r$ and the action of $A$ on $M$ are compatible. Firstly, we have
\begin{eqnarray}
\Phi(a(1_A \tens m))&=& \Phi(a\tens m) = (a_{[0]}m)_\beta \tens a_{[-1]}{}^\beta \nonumber\\
&=& (a_{[0]}m)_\beta \tens 1_H{}^\beta \ra a_{[-1]} \nonumber\\
&=& (a_{[0]}m)_{[0]}\tens (a_{[0]}m)_{[1]}\ra a_{[-1]}.\label{Phia}
\end{eqnarray}
In the above calculation, the last equality is just the definition $\rho^r$ from \eqref{eq:rho} and the third equality is justified as follows. Take $h\in H$ and define $f_h : H\to H$ by $f_h(g) = hg$. This map is left $H$-linear. Therefore, by naturality of $\beta^M$, we have
\[\xymatrixcolsep{5pc}\xymatrix{
H\tens M \ar[r]^-{\beta^M_H} \ar[d]_-{f_h \tens \Id_M} & M\tens C \ar[d]^-{\Id_M \tens E(f_h)}\\
H\tens M \ar[r]^-{\beta^M_H} & M\tens C
}\]
which, by evaluating in $1_H \tens m$ for all $m\in M$, gives $m_\beta\tens h^\beta = m_\beta \tens 1_H{}^\beta \ra h$. 

On the other hand, we have
\begin{equation}\label{aPhi}
a\Phi(1_A\tens m)=a(m_\beta \tens 1_H{}^\beta)=a_{[0]}m_\beta \tens a_{[1]}\la 1_H{}^\beta= a_{[0]}m_{[0]} \tens a_{[1]} \la m_{[1]}.
\end{equation}
Since $\Phi$ is $A$-linear, the expressions \eqref{Phia} and \eqref{aPhi} are equal, which leads exactly to the compatibility condition \eqref{eq:genYDcomp1} for generalized Yetter-Drinfeld modules, making $M \in {}_A\YD^C(H,K)$.

For the converse, let $M\in {}_A\YD^C(H,K)$. We define the half-braiding $\beta^M_V : V\tens M \to M\tens (C\tens_H V)$ given by \begin{equation}\label{defbeta}
\beta^M_V(v\tens m) = m_{[0]}\tens (m_{[1]}\tens_H v),
\end{equation}
which is natural in $V\in {}_H\Mod$. We show that it obeys the condition for half-braiding \eqref{eq:halfbraid}:
\begin{align*}
((\Id_M \tens \psi_{V,W})\circ \beta^M_{V\tens W})(v\tens w \tens m) =& m_{[0]}\tens \psi_{V,W}(m_{[1]}\tens_H (v\tens w)) \\
=& m_{[0]}\tens (m_{[1](1)} \tens_H v)\tens (m_{[1](2)} \tens_H w) \\
=& m_{[0]}\tens (m_{[0][1]} \tens_H v)\tens (m_{[1]} \tens_H w)\\
=&(\beta^M_V \tens \Id_{C\tens_H W})(v\tens m_{[0]} \tens (m_{[1]}\tens_H w))\\
=& ((\beta^M_V \tens \Id_{C\tens_H W})\circ (\Id_V \tens \beta^M_W))(v\tens w \tens m),
\end{align*}
where $\psi_{V,W}$ is given by \eqref{Eopmonoidal}. Therefore, we conclude that $\CZ_E(_A\Mod)\cong {}_A\YD^C(H,K)$. 
\end{proof}

\begin{example}
In case $H=K=A=C$, we recover the classical result recalled above that the category ${_H\YD^H}$ of left-right Yetter-Drinfeld modules over a bialgebra $H$ coincides with the weak center $\Zz^w({_H\Mod})$ of the category left modules over $H$. 
\end{example}

As before, we only suppose that $k$ is a commutative ring. Then for any $k$-bialgebra $H$, the category of $H$-comodules is again a monoidal category, and similarly the category of $C$-comodules over a $(K,H)$-comodule coalgebra is a $(\Mod^K,\Mod^H)$-biactegory. Furthermore, if $A$ is an $(H,K)$-bicomodule algebra, then it induces a monoidal functor $-\cotens^HA:\Mod^H\to \Mod^K$. Therefore, a YD datum $(H,K,A,C)$ also induces a dual center datum (see Remark~\ref{dualcenter}). 
However, in this dual setting we will make an additional restriction and only consider projective $k$-modules, in particular we will suppose that all modules over $k$-algebras and comodules over $k$-coalgebras are projective as $k$-module. If $k$ is a field this is no restriction. In this setting we obtain the dual result of Theorem~\ref{thm:mainthm}.

\begin{theorem}\label{YDiscenter}
Let $(H,K,A,C)$ be a Yetter-Drinfeld datum and $(\Mod^K,\Mod^H,\Mod^C,E=-\cotens^HA)$ the corresponding dual center datum, where we only consider projective $k$-modules. Then the lax $E$-center of the $(\Mod^K, \Mod^H)$-biactegory $\Mod^C$ of left $A$-modules is isomorphic to the category of generalized YD modules:
$$\CZ^w_{-\cotens^HA}(\Mod^C)\cong {}_A\YD^C(H,K) \cong \CZ^w_{C\ot_H-}({_A\Mod}).$$
\end{theorem}

\begin{proof}
Let us start by proving that that for any right $C$-comodule $M$, a natural transformation $\beta^M:-\cotens^H A\tens M\to M\tens -:{\Mod^H}\to \Mod^K$ is completely determined by the map
$$(\Id\tens \epsilon_H)\circ \beta^M_H: A\tens M\to M,$$
which we suggestively already denote as an action: $(\Id\tens \epsilon_H)\circ \beta^M_H(a\ot m)=am$. For any linear functional $w^*\in W^*$ on a right $H$-comodule $W$, we have now an $H$-colinear map $f_{w^*}: W\to H, f_{w^*}(w)=w^*(w_{[0]})w_{[1]}$. 
Since $\epsilon_H\circ f_{w^*}=w^*$, we then find
\begin{eqnarray*}
(\Id\tens w^*)\circ \beta^M_W &=& (\Id\tens \epsilon_H)\circ (\Id\tens f_{w^*})\circ \beta^M_W\\
&=& (\Id\tens \epsilon_H)\circ \beta^M_H\circ (f_{w^*}\ot \Id)
\end{eqnarray*}
where we used naturality of $\beta^V$ in the last equality. Hence we find for any $w\ot a\ot m\in W\cotens^HA\ot M$ and $w^*\in W^*$ that 
$$(\Id\ot w^*)\beta(w\ot a\ot m)=w^*(w)am.$$
Since $M$ and $W$ are projective as $k$-modules, the family of all maps $\Id\ot w^*$ is jointly monic. We then conclude that
$$\beta^M_W(w\ot a\ot m)=am\ot w.$$
As in the proof of Theorem~\ref{thm:mainthm}, the half-braid conditions on $\beta^M$ correspond to the associativity and unitality of the $A$-action on $M$ and the $C$-colinearity of $\beta^M$ translates to the YD condition on $M$.
\end{proof}

\begin{examples}
Let $H$ be a quasi-triangular Hopf algebra and $A$ a left $H$-comodule algebra. Then ${_H\Mod}$ is a braided monoidal category and ${_A\Mod}$ is a left ${_H\Mod}$-actegory. Then we can consider the reflexive center of ${_A\Mod}$ with respect to ${_H\Mod}$, which we know coincides with the $E$-relative center of ${_A\Mod}$ with respect to the identity functor $E:{_H\Mod}\to {_H\Mod}^{rev}$, where the monoidal structure is given by the braiding (see Remark~\ref{dualcenter}(4)). 
In \cite[Proposition 5.18]{Laugwitz} it was shown that the reflexive center of ${_A\Mod}$ with respect to ${_H\Mod}$ is isomorphic to a category of Doi-Hopf modules ${_A{\sf DH}^{\hat H}}(H)$, where $\hat H$ a suitably constructed module coalgebra over $H$, known as Majid's transmutation coalgebra. Taking into account the correspondence between generalized Yetter-Drinfeld modules and Doi-Hopf modules as described in \cite[Proposition 94]{BookStef}, this isomorphism of categories can be deduced from our main result above.

It has been proved in \cite{hassanzadeh} that the category of higher Yetter-Drinfeld modules ${}_H\YD^H(\Id, S^{2i})={_H\YD^{_{S^{2i}}H}}(H,H)$ is equivalent to the center of ${}_H\Mod$ regarded as a biactegory over itself with a modified action. This result is subsumed by Theorem~\ref{YDiscenter}.
\end{examples}

\subsection{Strong center, monoidal equivalences and Galois co-objects}
\label{se:Galois}

As in the previous section, let $C$ be a $(K,H)$-bimodule coalgebra and consider the associated op-monoidal functor $C\ot_H-:{_H\Mod}\to {_K\Mod}$. Recall the following result (see e.g.\ \cite{AV}, or \cite{Schauenburg:biGalois} for the dual case).
\begin{proposition}
Let $H$ and $K$ be bialgebras and $C$ be a $(K,H)$-bimodule coalgebra. Then the associated op-monoidal functor $C\ot_H-:{_H\Mod}\to {_K\Mod}$ is a monoidal equivalence if and only if $C$ is an $(K,H)$-bi-Galois co-object. The latter means that the canonical maps
\begin{eqnarray*}
\lcan: K\ot C\to C\ot C,&& \lcan(k\ot c)=k\la c_{(1)}\ot c_{(2)}; \quad \lcan_0: \kk\ot_K C \to \kk,\ \lcan_0 (\lambda \ot_K c)=\lambda \epsilon_C(c);\\
\rcan: C\ot H\to C\ot C,&& \rcan(c\ot h)=c_{(1)}\ot c_{(2)}\ra h; \quad \rcan_0: C\ot_H \kk \to \kk,\ \rcan_0 (c\ot_H \lambda)=\epsilon_C(c)\lambda.
\end{eqnarray*}
are bijective.
\end{proposition}

The quasi-inverse of the monoidal equivalence $C\ot_H-$ is again of the form $\ol C\ot_K-:{_K\Mod}\to {_H\Mod}$ for some $(H,K)$-bimodule coalgebra $\ol C$. In case $H$ and $K$ have bijective antipodes, $\ol C$ is given by the co-opposite coalgebra of $C$, endowed with actions
$$h\bla c\bra k:=S^{-1}_K(k) \la c\ra S^{-1}_H(h).$$

The functors $C\ot_H-$ and $\ol C\ot_K-$ being inverse equivalences, there is a strict Morita context $(H,K,\ol C,C,\wedge,\vee)$. The Morita maps can be expressed in terms of the bi-Galois co-object structures by the following formula:
\begin{eqnarray}
\wedge: \ol C\ot_K C\to H, &\quad & \ol c\wedge c = (\epsilon_C\ot \Id_H)\circ \rcan^{-1}(\ol c \ot c);\label{wedge}\\
\vee: C\ot_H \ol C\to K, && c\vee \ol c = (\Id_K\ot \epsilon_C)\circ \lcan^{-1}(c\ot \ol c). \label{vee}
\end{eqnarray}
These maps satisfy in particular the following identities:
\begin{eqnarray}
c_{(1)}\wedge c_{(2)}=&\epsilon(c)&=c_{(1)}\vee c_{(2)}; \label{wedgeveeepsilon}\\
c\ra (d\wedge e)&=&(c\vee d)\la e. \label{wedgeveeass}
\end{eqnarray}
The bijectivity of $\ul\can_0$ (or $\ol\can_0$) furthermore guarantees the existence of an element $u\in C$ such that $\epsilon_C(u)=1$. Then $\{(u_{(1)}, u_{(2)}\wedge- )\}$ is a dual base for $C$ as right $H$-module and $\{(u_{(2)}, -\vee u_{(1)})\}$ is a dual base of $C$ as left $K$-module.

The Morita maps $\wedge$ and $\vee$ are moreover coalgebra maps. Therefore we find that the element  $\vee^{-1}(1_K)=u_{(1)}\ot_H u_{(2)}\in C\ot \ol C$ is grouplike, and therefore we have the identity
\begin{equation}
(u_{(1)}\ot_H u_{(4)})\ot (u_{(2)}\ot_H u_{(3)}) = (u_{(1)}\ot_H u_{(2)}) \ot (u_{(1')}\ot_H u_{(2')}) \label{ugrouplike}
\end{equation}
From these identity, it follows that the map
\begin{equation}
\ol\sigma: C\to \ol C,\ \ol\sigma(c)= S^{-1}_H(u_{(1)}\wedge c)\bla u_{(2)}\label{sigma}
\end{equation}
is a coalgebra map that mimics some properties of the inverse antipode of $H$. In particular, we have the following identity
\begin{equation}\label{sigmaantipode}
\ol\sigma(c_{(2)})\wedge c_{(1)} = \epsilon_C(c)1_H
\end{equation}
Furthermore, $\ol\sigma$ respects the actions in the following way.
\begin{equation}
\ol\sigma(k\la c\ra h)=S^{-1}_H(h)\bla \ol\sigma(c) \bra S^{-1}_K(k) \label{sigmalinear}
\end{equation}
For more details about the above constructions we refer to \cite{AV}.\\

Our next aim is to provide a characterization of bi-Galois co-objects in terms of the relative center. However, before doing this, we recall the notion of a totally faithful module and some of its properties, which will turn out to be useful in what follows.

Let $R$ be any ring. Recall from \cite[Definition 2.1]{CDV} that a right $R$-module $P$ is called {\em totally faithful} if for any left $R$-module $M$, we have $m=0$ in $M$ whenever $p\ot_R m=0$ in $P\ot_R M$ for all $p\in P$. In other words, $P$ is totally faithful if and only if the map 
$$\eta_N: N\to {\sf Ab}(P,P\ot_R N),\quad n\mapsto (p\mapsto p\ot_R n)$$ is injective for all $N\in {_R\Mod}$. Remark that $\eta_N$ is the unit of the Hom-tensor adjunction for $P$. Faithfully flat modules are clearly totally faithful. 

On the other hand, a right $R$-module morphism $f:N\to P$ is called {\em pure} if for any left $R$-module $M$, the map $f\ot_R \Id_R : N\ot_R M\to P\ot_R M$ is injective. In particular $f$ itself is injective and  we also say that $N\cong f(N)$ is a pure sub-module of $P$. Recall from \cite[Lemma 2.2]{CDV} that both notions are related in the way that an $(S,R)$-bimodule $P$ that is finitely generated and projective as left $S$-module is totally faithful as right $R$-module if and only if the canonical $R$-linear map $R\to {_S\End}(P)$ is pure.

The following result is also implicit in \cite{CDV} and \cite{mesablishvili}.

\begin{lemma}\label{totallyfaithful}
If $P$ is totally faithful as right $R$-module, then the functor $P\ot_R-: {_R\Mod}\to {\sf Ab}$ reflects monomorphisms and isomorphisms.
\end{lemma}

\begin{proof}
First, let $f:N\to M$ be a map such that $P\ot_R f:P\ot_R N\to P\ot_R M$ is injective. Then take any $n\in \ker f$ and any $p\in P$. We find that $p\ot_R f(n)=0$ for all $p\in P$ and by injectivity of $P\ot_R f$ this means $p\ot_R n=0$ for all $p\in P$. Then, by total faithfulness, we find $n=0$.

Now suppose that $P\ot_R f$ is an isomorphism. We have to prove that $f$ is surjective. Consider the following diagram where the undecorated arrows on the lower row are given by applying the functor $\Hom(P,P\ot_R-)$ to the corresponding maps in the upper row.
\[\xymatrix{
N \ar[rr]^-f \ar@{^(->}[d]_{\eta_N} && M \ar[rr]^-\pi \ar@{^(->}[d]_{\eta_M} && M/f(N) \ar@{^(->}[d]_{\eta_{M/f(N)}} \ar[r] & 0\\
\Hom(P,P\ot_R N) \ar[rr] && \Hom(P,P\ot_R M) \ar[rr] && \Hom(P,P\ot_R (M/f(N)))
}\]
The upper row is exact by construction and therefore the composition of both maps in the lower row is the zero map. By total faithfulness, we know that the vertical maps are injective. Since $P\ot f$ is an isomorphism, $\Hom(P,P\ot_R f)$ is an isomorphism as well. Therefore, the image of the map $\Hom(P,P\ot_R \pi)$ must be zero. By commutativity of the diagram and the surjectivity of $\pi$, this means that the image of $\eta_{M/f(N)}$ is zero as well and hence $M/f(N)$ is zero as $\eta$ is always injective. Since $M/f(N)$ is zero, it follows that $M=f(N)$, and thus $f$ is surjective.
\end{proof}

\begin{theorem}
As in Theorem~\ref{thm:mainthm}, let $(H,K,A,C)$ be a Yetter-Drinfeld datum and consider the corresponding center datum $({_H\Mod},{_K\Mod},{_A\Mod},E=C\ot_H-)$. Then the following statements hold.
\begin{enumerate}
\item If $\rcan$ is an isomorphism then the lax $E$-center of ${_A\Mod}$ coincides with the strong $E$-center. 
$$\Zz^w_{C\ot_H-}(_A\Mod) = \Zz_{C\ot_H-}(_A\Mod)\cong {_A\YD^C}(H,K)$$
\item Conversely, if  the lax $E$-center coincides with the strong $E$-center and $A$ is totally faithful as $k$-module, or $H$ and $C$ are flat as left $k$-module and ${^{co K}A}\cong k$, then $\rcan$ is an isomorphism. 
\end{enumerate}
\end{theorem}

\begin{proof}
If $\rcan$ is an isomorphism, then we define for each $M\in {}_A\YD^C(H,K)$ and each $V\in {}_H\Mod$ a map
$$\ol\beta^M_V: M\ot C\ot_H V\to V\ot M,\quad \ol\beta^M_V(m\ot c\ot_H v)=(m_{[1]}\wedge c)\la v\ot m_{[0]}$$
where $m\in M$, $c\in C$, $v\in V$ and we used the notation of \eqref{wedge}. Let us check that $\ol\beta^M_V$ is a two-sided inverse for $\beta^M_V$ defined in \eqref{defbeta}.

\begin{eqnarray*}
\ol\beta^M_V\circ \beta^M_V(v\ot m)&=&\ol\beta^M_V(m_{[0]}\ot m_{[1]}\ot_H v)
= (m_{[1]}\wedge m_{[2]})\la v\ot m_{[0]}\\
&\stackrel{\eqref{wedgeveeepsilon}}{=}& \epsilon(m_{[1]})v\ot m_{[0]}= v\ot m\\
\beta^M_V\circ \ol \beta^M_V(m\ot c\ot_H v)
&=& \beta^M_V((m_{[1]}\wedge c)\la v\ot m_{[0]})
= m_{[0]}\ot m_{[1]}\ot_H (m_{[2]}\wedge c)\la v\\
&=& m_{[0]}\ot m_{[1]}\ra (m_{[2]}\wedge c)\ot_Hv
\stackrel{\eqref{wedgeveeass}}{=} m_{[0]}\ot (m_{[1]}\vee m_{[2]})\la c\ot_Hv\\
&\stackrel{\eqref{wedgeveeepsilon}}{=}& m_{[0]} \ot \epsilon(m_{[1]})c\ot_H v
= m\ot c\ot_H v
\end{eqnarray*}

Conversely, consider the YD module $A\ot C$ with actions and coactions given by
\begin{eqnarray*}
a'\cdot (a\ot c)&=&a'_{[0]}a\ot a'_{[1]}\la c \ra S^{-1}(a'_{[-1]})\\
\rho(a\ot c)&=& a\ot c_{(1)}\ot c_{(2)}.
\end{eqnarray*}
Furthermore, consider the regular $H$-module, then we find that
$$\beta^{A\ot C}_H : H\ot A\ot C\to A\ot C \ot C\ot_H H \cong A \ot C\ot C,\quad \beta^{A\ot C}_H (h\ot a\ot c)=a\ot c_{(1)}\ot c_{(2)}\ra h.$$
This means $\beta^{A\ot C}_H = (\Id_A\ot \rcan)\circ (\mathrm{flip}\ot \Id_C)$ where here $\mathrm{flip} : H\tens A \to A\tens H$ is the symmetry on $k$-modules, and hence if $\beta$ is invertible then $\Id_A\ot \rcan$ is invertible, which implies that $\rcan$ is invertible whenever $A$ is totally faithful as $k$-module (see Lemma \ref{totallyfaithful}). On the other hand, if $C$ and $H$ are flat as left $k$-modules, then ${^{co K}(A\ot C\ot H)} \cong ({^{co K}A})\ot C\ot H$. The later is isomorphic to $C\ot H$ if moreover the $K$-coinvariants of $A$ are trivial. Similarly we find under the same conditions that ${^{co K}(A\ot C\ot H)}\cong C\ot C$. Therefore, applying the functor that takes $K$-coinvariants to the isomorphism $\Id_A\ot \rcan$, we obtain that $\rcan$ is an isomorphism.
\end{proof}

For sake of completeness, let us also state the dual result, where we again restrict to the case that all considered $k$-modules are projective. 

\begin{theorem}
Let $(H,K,A,C)$ be a Yetter-Drinfeld datum and $(\Mod^K,\Mod^H,\Mod^C,E=-\cotens^HA)$ the corresponding dual center datum, where we only consider projective $k$-modules.\\
Then the following statements hold.
\begin{enumerate}  
\item If 
$\lcan_A:  A\ot A\to H\ot A$ given by $\lcan_A(a\ot b)=a_{[-1]}\ot a_{[0]}b$ is an isomorphism, then the lax $E$-center of ${\Mod^C}$ coincides with the strong $E$-center. 
$$\Zz^w_{-\cotens^HA}(\Mod^C) = \Zz_{-\cotens^HA}(\Mod^C)\cong {_A\YD^C}(H,K)$$
\item Conversely, if  the lax $E$-center coincides with the strong $E$-center and $C$ is totally faithful as $k$-module or $\rcan_0$ is bijective, then $\lcan_A$ is an isomorphism. 
\end{enumerate}
\end{theorem}

Finally it is worthwhile to remark that it can happen that $C$ is a Galois co-object, while $A$ is not a Galois object. In such a case, the weak center $\Zz^w_{-\cotens^HA}(\Mod^C)$ coincides with the strong center $\Zz_{C\ot_H-}(_A\Mod)$, while these are strictly bigger than the strong center $\Zz_{-\cotens^HA}(\Mod^C)$, as otherwise $A$ would be Galois as well.

\subsection{Braided structure on generalized Yetter-Drinfeld modules}
\label{braidingYD}

Consider YD data $(H,K,A,C)$ and $(K,L,B,D)$. Then 
the tensor product $D\tens_K C$ is an $(L,H)$-bimodule coalgebra with left and right action
\begin{align}
\la:& L\tens D\tens_K C \to D\tens_K C, \qquad l \la (d\tens_K c) = (l\la d)\tens_K c\label{eq:calc3}\\
\ra:& D\tens_K C \tens H \to D\tens_K C, \qquad (d\tens_K c)\ra h = d\tens_K (c\ra h)\label{eq:calc4}
\end{align}
for all $d\tens_K c \in D\tens_K C, h\in H$, and $l\in L$. 
If furthermore the cotensor product $A\cotens^K B$ is a pure $k$-submodule of $A\ot B$ (e.g. when $k$ is a field or when $A=K$) , then $A\cotens^K B$ is an $(H,L)$-bicomodule algebra with left and right coaction
\begin{align}
(a\tens b)_{[-1]}\tens (a\tens b)_{[0]}  =& a_{[-1]}\tens (a_{[0]}\tens b) \in H \tens A\cotens^K B \label{eq:calc1} \\
(a\tens b)_{[0]}\tens (a\tens b)_{[1]}  =& (a\tens b_{[0]})\tens b_{[1]} \in A\cotens^K B \tens L \label{eq:calc2} 
\end{align}
for all $a\tens b \in A\cotens^K B$. This makes $(H,L,A\cotens^K B, C\tens_K D)$ a YD datum. 

\begin{proposition}\label{prop:MonoidalGenYD}
With notation and under conditions as stated above, the usual tensor product over the underlying base $\kk$ lifts strictly to a bifunctor
\[{}_A\YD^C(H,K)\times {}_B\YD^D(K,L)\to {}_{A\cotens^K B}\YD^{D\tens_K C}(H, L).\]
\end{proposition}

\begin{proof}
Let $M \in {}_A\YD^C(H,K)$ and $N\in {}_B\YD^D(K,L)$. We will show that 
the identity \eqref{eq:genYDcomp1} holds for all $a\tens b \in A\cotens^K B$ and $m\tens n \in M\tens N$, so that we can conclude that
$M\tens N$ is in ${}_{A\cotens^K B}\YD^{D\tens_K C}(H, L)$. We note that $M\tens N$ is a right-comodule over $D\tens_K C$ and a left-module over $A\cotens^K B$ with the following coaction and action
\begin{align}
\cdot :& A\cotens^K B \tens M\tens N \to M\tens N, \qquad (a\tens b)(m\tens n)=am \tens bn \label{eq:calc5}\\
\rho^r:& M\tens N \to M\tens N \tens D\tens_K C, \qquad \rho^r(m\tens n)= m_{[0]}\tens n_{[0]}\tens n_{[1]}\tens_K m_{[1]}\label{eq:calc6}
\end{align}
for all $a\tens b \in A\cotens^K B$ and $m\tens n \in M\tens N$. We compute
\begin{align*}
& \hspace{-.5cm}\big[(a\tens b)_{[0]} (m\tens n)\big]_{[0]} \tens \big[(a\tens b)_{[0]} (m\tens n)\big]_{[1]}\ra (a\tens b)_{[-1]}\\
\stackrel{1}{=}&\big[(a_{[0]}\tens b)(m\tens n) \big]_{[0]} \tens \big[(a_{[0]}\tens b)(m\tens n) \big]_{[1]}\ra a_{[-1]}
\stackrel{2}{=} (a_{[0]}m\tens bn)_{[0]} \tens (a_{[0]}m\tens bn)_{[1]} \ra a_{[-1]}\\
\stackrel{3}{=}& (a_{[0]}m)_{[0]} \tens (bn)_{[0]} \tens (bn)_{[1]} \tens_K \big((a_{[0]}m)_{[1]} \ra a_{[-1]})\big)
\stackrel{4}{=} a_{[0]}m_{[0]} \tens (bn)_{[0]} \tens (bn)_{[1]} \tens_K a_{[1]}\la m_{[1]}\\
\stackrel{5}{=}& am_{[0]}\tens (b_{[0]} n)_{[0]}\tens (b_{[0]} n)_{[1]} \ra b_{[-1]} \tens_K m_{[1]}
\stackrel{6}{=} am_{[0]} \tens b_{[0]}m_{[0]} \tens b_{[1]}\la n_{[1]} \tens_K m_{[1]}\\
\stackrel{7}{=}& (a\tens b)_{[0]}(m\tens n)_{[0]} \tens (a\tens b)_{[1]} \la (m\tens n)_{[1]}.
\end{align*}
We use the right coaction of $L$ on $A\cotens^K B$ as defined in \eqref{eq:calc1} for the first equality, the left action of $A\boxtimes^K B$ on $M\tens N$ as defined in \eqref{eq:calc5} for the second equality, and the right coaction of $D\tens_K C$ on $M\tens N$ as defined in \eqref{eq:calc6} followed by the right $H$ action on $D\tens_K K$ as defined in \eqref{eq:calc4} for the third equality. We then use \eqref{eq:genYDcomp1} applied on the element $am$ in the YD module $M$ to get the fourth equality. Then by using the identity $a_{[0]}\tens a_{[1]}\tens b = a\tens b_{[-1]}\tens b_{[0]}$ (because $a\tens b \in A\cotens^K B$), in combination with the $K$-balancedness of the last tensor product, we obtain the fifth equality. The sixth equality follows by applying \eqref{eq:genYDcomp1} to the element $bn$ in the YD module $N$. Finally, the last equality follows again by the definitions of the various actions and coactions as defined in \eqref{eq:calc5}, \eqref{eq:calc3}, \eqref{eq:calc2} and \eqref{eq:calc6}.
\end{proof}

Proposition \ref{prop:MonoidalGenYD} above leads in particular to bi-functors
\begin{eqnarray*}
{}_H\YD^H \times {}_A\YD^C(H,K) &\to& {}_{H\cotens^H A}\YD^{C\tens_H H}(H, K) \cong  {}_A\YD^C(H, K)\\
{}_A\YD^C(H, K) \times {}_K\YD^K &\to& {}_{A\cotens^K K}\YD^{K\tens_K C}(H, K) \cong {}_A\YD^C(H, K).
\end{eqnarray*}
By means of this functors, ${}_A\YD^C(H,K)$ obtains the structure of a $({}_H\YD^H,  {}_K\YD^K$)--biactegory, which also follows from Lemma~\ref{biactlift} in combination with Theorem~\ref{thm:mainthm}. In particular it encodes the monoidal structure on usual YD modules, and the actegory structure on AYD modules. Our next aim is to understand its braiding.

\begin{lemma}\label{generalbetaproperties}
For any YD-datum $(H,K,A,C)$, any Yetter-Drinfeld module $M\in {_A\YD^C}(H,K)$ and any left $H$-module $V\in {_H\Mod}$, the map 
\[\beta_{V,M}: V\tens M \to M \tens (C\tens_H V), \quad \beta_{V,M}(v\tens m) = m_{[0]}\tens m_{[1]}\tens_H v\]
is left $A$-linear.

Furthermore, for any additional YD-datum $(K,L, A', C')$, any YD-module $N\in {_{A'}\YD^{C'}(K,L)}$ and any left $H$-module $U\in {_H\Mod}$, we have the following identities
\begin{eqnarray}
(\Id_M\ot \psi^C_{V,M})\circ \beta_{U\ot V,M}&=&(\beta_{U,M}\ot \Id_{C\ot_H V})\circ (\Id_U\ot \beta_{V,M})
\label{coherencebeta2}\\
\beta_{V,M\ot N}&=&(\Id_M\ot \beta_{C\ot_H V,N})\circ (\beta_{V,M}\ot \Id_N) \label{coherencebeta1}
\end{eqnarray}

\[\xymatrixcolsep{5pc}\xymatrix{
&  M\ot C\ot_H(U\ot V) \ar[dr]^-{\Id_M\ot \psi^C_{V,M}} \\
U\ot V \ot M \ar[dr]_-{\Id_U\ot \beta_{V,M}} \ar[ur]^-{\beta_{U\ot V,M}} && M\ot C\ot_H U \ot C\ot_H V\\
& U\ot M\ot C\ot_H V \ar[ur]_-{{}\hspace{.8cm}\beta_{U,M}\ot \Id_{C\ot_H V}}
}\]

\[\xymatrixcolsep{5pc}\xymatrix{
V\tens M \tens N \ar[dr]_-{\beta_{V,M}\tens \Id_N} \ar[rr]^-{\beta_{V, M\tens N}} &&  M\tens N \tens C'\tens_K C \tens_H V\\
& M\tens C\tens_H V \tens N \ar[ur]_-{{}\hspace{.8cm}\Id_M\tens \beta_{C\tens_H V, N}} \\
}
\]
\end{lemma}

\begin{proof}
The left $A$-linearity of $\beta_{V,M}$ follows from the following direct computation.
\begin{eqnarray*}
\beta_{V,M}(a\cdot (v\tens m))&=& \beta_{V,M}(a_{[-1]}\la v\tens a_{[0]}m)\\
&=& (a_{[0]}m)_{[0]} \ot (a_{[0]}m)_{[1]}\ot_H a_{[-1]}\la v \\
&=&(a_{[0]}m)_{[0]} \ot (a_{[0]}m)_{[1]}\ra a_{[-1]}\ot_H v\\
&=& a_{[0]}m_{[0]} \ot a_{[1]}\la m_{[1]} \ot _H v\\
&=& a\cdot(m_{[0]}\ot m_{[1]}\ot_H v)
\end{eqnarray*}
where we used the YD compatibility \eqref{eq:genYDcomp1} for $M$ in the fourth equality.

The identity \eqref{coherencebeta1} follows from the $C'\ot_K C$-comodule structure on $M\ot N$ as defined in equation \eqref{eq:calc6} in the proof of Proposition \ref{prop:MonoidalGenYD}.

The identity \eqref{coherencebeta2} follows from the mixed coassociativity for the $C$-comodule $M$.
\end{proof}

Consider the morphism $\beta_{V,M}$ defined in Lemma~\ref{generalbetaproperties}, and suppose moreover that $V$ is a YD-module over $H$. Then the domain of $\beta_{V,M}$ is again a generalized YD-module, and it would be natural to expect that $\beta$ becomes a morphism of YD-modules. However, in general one cannot expect that the codomain of $\beta_{V,M}$ is a YD-module as well. For this to happen, we would need $C\ot_H V$ to be a YD-module over $K$, so for the functor $C\ot_H-:{_H\Mod}\to {_K\Mod}$ to lift to YD-modules. For this reason, in the next Lemma, we restrict to the case where $C$ is a bi-Galois co-object, as this allows us to apply Lemma~\ref{liftingtocenter}. 

\begin{corollary}\label{GaloisCoobjectlifts}
Let $H$ and $K$ be bialgebras, $C$ a $(K,H)$-bi-Galois co-object and $D$ an $(H,H)$-bimodule coalgebra. Then there is an equivalence of categories
$$C\ot_H- : {_H\YD^D}(H,H)\to {_K\YD^{C\ot_HD\ot_H\ol C}}(K,K)$$
In particular, in case $D=H$, we obtain an equivalence of categories
$$C\ot_H- : {_H\YD^H}\to {_K\YD^K}$$
\end{corollary}

\begin{proof}
This is a direct consequence of Lemma~\ref{liftingtocenter} in combination with Theorem~\ref{YDiscenter}, taking into account that a $(K,H)$-bi-Galois co-object $C$ induces a monoidal equivalence $C\ot_H-:{_H\Mod}\to {_K\Mod}$, as recalled at the beginning of the previous subsection.

More explicitly, take $V\in {_H\YD^D}(H,H)$. Then the left $K$-module $C\ot_H V$ can be endowed with a right $C\ot_HD\ot_H\ol C$-comodule structure via
\begin{equation}\label{coactionliftYD}
\rho(c\ot_H v)=(c_{(2)}\ot_H v_{[0]})\ot c_{(3)}\ot_H v_{[1]}\ot_H \ol\sigma(c_{(1)}),
\end{equation}
where $\ol\sigma:C\to \ol C$ is the coalgebra map defined in \eqref{sigma}.
Let us check that the YD compatibility indeed holds. For any $k\in K$ and $c\ot_H v\in C\ot_H V$, we find
\begin{eqnarray*}
\rho(k\la c\ot_H v))
&=&(k_{(2)}\la c_{(2)}\ot_H v_{[0]})\ot k_{(3)}\la c_{(3)}\ot_H v_{[1]}\ot_H \ol\sigma(k_{(1)}\la c_{(1)})\\
&=&(k_{(2)}\la c_{(2)}\ot_H v_{[0]})\ot k_{(3)}\la c_{(3)}\ot_H v_{[1]}\ot_H \ol\sigma(c_{(1)})\bra S^{-1}_K(k_{(1)}),
\end{eqnarray*}
where we used \eqref{sigmalinear} in the second equality. This shows that \eqref{eq:genYDcomp2} is indeed satisfied.
\end{proof}

\begin{remark}
It is worth to note that the condition that $C$ is a bi-Galois co-object, is used for the functor $E\simeq C\ot_H-$ to lift to Yetter Drinfeld modules via Corollary~\ref{GaloisCoobjectlifts}. However, the Galois condition is sufficient but not necessary for such a lifting to exist. The sufficient and necessary condition, is the existence of a ``YD-entwining map'' $\alpha: C\ot H\to K\ot C$, satifsying suitable conditions which can be found in \cite{AV}. The above corollary holds as well if we reduce the Galois condition on $C$ by the condition that a YD entwining map exists. However, because of a lack of examples we avoid the discussion of YD entwining maps here, and restrict ourselves to bi-Galois co-objects.
\end{remark}

We are now ready to state and prove the braided structure on the category of generalized Yetter-Drinfeld modules. We can apply Theorem \ref{thm:centerbiact} directly to our current setting but in fact, we have even a more general statement.

\begin{theorem}\label{thm:generalbraid}
Let $(H,K,A,C)$ and $(H,H,H,D)$ be YD-data, where $C$ is a bi-Galois co-object. Then for any Yetter-Drinfeld modules $M\in {_A\YD^C}(H,K)$ and $V\in {_H\YD^D(H,H)}$,
the map $\beta_{V,M}$ defined in Lemma~\ref{generalbetaproperties} is also right $C\ot_HD$-colinear. The situation is summarized in the following diagram
\[
\xymatrix{
{_H\YD^D}(H,H) \quad \times \ar[ddr]_(.30){C\ot_H-} & \hspace{-1cm} {_A \YD^C}(H,K)  \ar[dr]^\ot \ar[ddl]^(.30){\Id} \\
& \qquad\qquad \stackrel{\beta}{\Rightarrow}& {_A\YD^{C\ot_H D}}(H,K) \\
{_A \YD^C}(H,K) \quad \times & \hspace{-1cm} {_K\YD^{C\ot_H D\ot_H \ol C}}(K,K)\ar[ur]_\ot
}\]

In particular, ${_A\YD^C}(H,K)$ is a $C\ot_H-
$-braided $({_H\YD^H}, {_K\YD^K})$-biactegory.
\end{theorem}

\begin{proof}
First remark that by Proposition \ref{prop:MonoidalGenYD}, the tensor product of an object of ${_A\YD^C}(H,K)$ and an object of ${_K\YD^{C\ot_HD\ot_H\ol C}}(K,K)$ lives in the category ${_A\YD^{C\ot_HD}}(H,K)$ because of the isomorphism $\wedge:\ol C\ot_K C \to H$. Furthermore, recall the formula for the right $(C\ot_H D\ot_H \ol C)$-coaction of $C\ot_H V$ from \eqref{coactionliftYD}.
Having all this in mind, let us check that $\beta_{V,M}:V\ot M\to M\ot C\ot_H V$ is right $C\ot_H D$-colinear:
\begin{eqnarray*}
\rho_{M\ot C\ot_H V}(m_{[0]}\ot m_{[1]}\ot_H v) &=&
m_{[0]} \ot m_{[3]}\ot_H v_{[0]} \ot m_{[4]} \ot_H v_{[1]} \ra (\ol\sigma(m_{[2]})\wedge m_{[1]})\\
&=& m_{[0]}\ot m_{[1]}\ot_H v_{[0]} \ot m_{[2]}\ot_H v_{[1]}\\
&=&\beta_{V,M}(v_{[0]}\ot m_{[0]})\ot m_{[1]}\ot_H v_{[1]}
\end{eqnarray*}
where we used \eqref{sigmaantipode} in the second equality.

Specializing to the case $D=H$, we find that $C\ot_H H\ot_H \ol C\cong C\ot_H\ol C\cong K$, so that the image of the functor $C\ot_H-$ lies in ${_K\YD^K}(K,K)$. By \eqref{coherencebeta2}, we have that the axiom \eqref{prebraidedact1} is satisfied. Furthermore, by choosing $(M,N)$ as $(V,M)$ in \eqref{coherencebeta1} we recover the left hand side of axiom \eqref{prebraidedact2}. Similarly, by choosing $(M,N)$ as $(M,X)$ in \eqref{coherencebeta1} we recover the right hand side of axiom \eqref{prebraidedact2}. We conclude that
\begin{eqnarray*}
    \beta_{U,V\tens M} &=& (\Id_V \tens \beta_{C\tens_H U, M})(\beta_{U,V}\tens \Id_M)\\
    (\Id_M \tens \psi_{U,V})\beta_{U\tens V, M} &=&(\beta_{U,M}\tens \Id_{C\tens_H V})(\Id_U \tens \beta_{V,M})\\
    \beta_{V,M\tens X}&=& (\Id_M \tens \beta_{C\tens_H V,X})(\beta_{V,M}\tens \Id_X)
\end{eqnarray*}
for any $U\in {}_H\YD^H$ and $X\in {}_K\YD^K$, making it a $C\tens_H-$-pre-braiding.
\end{proof}

Let us finish this section by stating a summary of the dual results. 

\begin{theorem}\label{thm:generalbraiddual}
Let $(A,C,H,K)$ and $(B,H,H,H)$ be YD-data, $M\in {_A\YD^C}(H,K)$ a Yetter-Drinfeld module and $V\in \Mod^H$ any comodule.
\begin{enumerate}
\item The map $\beta_{M,V}:M\ot (V\cotens^HA)\to V\ot M,\ \beta_{V,M}(m\ot v\ot a)=v\ot am$ is right $C$-colinear and satisfies braid relations similar to those of Lemma~\ref{generalbetaproperties}.
\item If moreover $A$ is a bi-Galois object, then there is an equivalence of categories $-\cotens^H A:{_B\YD^H}(H,H)\to {_{\ol A\cotens^H B\cotens^H A}\YD^K}(K,K)$.
\item If moreover $V\in {_B\YD^H(H,H)}$ is a Yetter Drinfeld module, then $\beta_{M,V}$ is also left $B\ot A$-linear and hence a morphism in ${_{B\cotens^H A}\YD^C}(H,K)$.
\item Knowing that $-\cotens^H A$ is an equivalence, we can take $W=V\cotens^HA$ and $B'=\ol A\cotens^H B\cotens^H A$ and can rewrite $\beta$ in an equivalent way as a map $\beta'_{M,W}:M\ot W\to W\cotens^H\ol A\ot M$ in ${_{B'\cotens^H \ol A}\YD^C}(H,K)$. This situation is summarized in the next diagram.
\end{enumerate}
\[
\xymatrix{
{_A\YD^C}(H,K) \quad \times \ar[ddr]_(.30){\Id} & \hspace{-1cm} {_{B'} \YD^K}(K,K)  \ar[dr]^\ot \ar[ddl]^(.30){-\cotens^K \ol A} \\
& \qquad\qquad \stackrel{\beta'}{\Rightarrow}& {_{A\cotens^H B'}\YD^{C}}(H,K) \\
{_{A\cotens^K B'\cotens^K \ol A} \YD^H}(H,H) \quad \times & \hspace{-1cm} {_A\YD^C}(H,K)\ar[ur]_\ot
}\]
In particular, ${_A\YD^C}(H,K)$ is a $-\cotens^HA$-braided $({_H\YD^H}, {_K\YD^K})$-biactegory.
\end{theorem}

\section{Bicategorical perspective: double groupoid-crossed bicategories}\label{se:bicats}

The aim of this section is to provide a bicategorical interpretation of the results of the previous session and to establish a link with crossed categorical structures in the sense of Turaev.

\subsection{Turaev's group-crossed braided monoidal categories}

Let us start by recalling the following notion, which appeared first in \cite{TuraevArxiv}, see also \cite{TuraevBook} for the published version.

Let $G$ be a group. A {\em right $G$-crossed braided monoidal category} is a category $\Cc$, such that
\begin{itemize}
\item $\Cc$ is $G$-graded, meaning that it decomposes as a disjoint union of full subcategories $\Cc=\bigsqcup_{g\in G} \Cc_g$.
In other words, we have an assignment $\partial:\Cc\to G$ such that, for each non-zero morphism $f:X\to Y$ in $\Cc$, we have $\partial X=\partial f =\partial Y$.
\item $\Cc$ has a (strict) monoidal structure that is compatible with the grading in the following way:
$$\ot: \Cc_g \times \Cc_h \to \Cc_{gh},$$
stated otherwise
$$\partial(X\ot Y)=\partial X\partial Y.$$
This means that $\partial$ is a monoidal functor, where $G$ is regarded as a discrete monoidal category.
\item $G$ acts (from the right) on $\Cc$ by means of monoidal automorphisms. In other words, there is a group morphism $\phi:G^{op}\to \Aut^{\ot}(\Cc),\ h\mapsto \phi(h)=(-)^h$, where $\Aut^{\ot}$ denotes the group of strict monoidal automorphisms of $\Cc$. This means that for all $X,Y$ in $\Cc$, and $h,g\in G$ we have identities
$$(X\ot Y)^h=X^h\ot Y^h, \qquad (X^h)^g=X^{hg}.$$
\item
The morphism $\phi$ is compatible with the grading in the sense that $\partial(X^h)=h^{-1}\partial(X)h$ for all $h\in G$ and $X\in \Cc$, in other words for all $g,h\in G$, we have
$$(-)^h : \Cc_g\to \Cc_{h^{-1}gh}.$$
\item There is braiding of the following kind:
For $X\in \Cc_g$ and $Y\in \Cc_h$, we have
$$\beta_{X,Y} : X\ot Y \to Y \ot X^{\partial Y} \in \Cc_{gh},$$
satisfying the following expected ``hexagon''\footnote{Because we omitted the associtivity constraints, we only see 3 instead of 6 arrows in these diagrams} conditions.
\[
\xymatrix{
X\ot Y \ot Z \ar[rr]^-{\Id_X\ot \beta_{Y,Z}} \ar[d]_-{\beta_{(X\ot Y),Z}} && X\ot Z\ot Y^{\partial Z} \ar[d]^-{\beta_{X,Z}\ot \Id_{Y^{\partial Z}}}\\
Z\ot (X\ot Y)^{\partial Z} \ar@{=}[rr] && Z\ot X^{\partial Z}\ot Y^{\partial Z}
}\quad
\xymatrix{
X\ot Y \ot Z \ar[rr]^-{\beta_{X,Y}\ot \Id_Z} \ar[d]_-{\beta_{X, (Y\ot Z)}} && Y\ot X^{\partial Y}\ot Z \ar[d]^-{\Id_Y \ot \beta_{X^{\partial Y},Z}}\\
Y\ot Z\ot X^{\partial(Y\ot Z)} \ar@{=}[rr] && Y\ot Z\ot X^{\partial Y\partial Z}
}
\]
\end{itemize}
Similarly, one defines a {\em left} $G$-crossed braided monoidal category, which is again a $G$-graded monoidal category, where now $G$ acts on the left on $\Cc$ by monoidal automorphisms, via a group morphism $G\to \Aut^\ot(\Cc), h\mapsto {^h(-)}$ and where the braiding is of the form $\beta_{X,Y}: X\ot Y\to {^{\partial X}Y}\ot X$. 

\begin{example}\label{ex:Panaite}
Let $H$ be a Hopf algebra, and consider the group $G=\Aut_{\sf Hopf}(H)$ of Hopf algebra automorphisms of $H$. Denote by ${^\alpha H}$ the bicomodule algebra obtained from $H$ by twisting the left regular $H$-coaction with $\alpha$, see Section~\ref{se:alphabeta}. One can observe that ${^\alpha H}$ is in fact an $H$-bi-Galois object, and since ${^\alpha H}\cotens^H {^\beta H}\cong {^{\alpha\beta} H}$, the assignment $\Aut_{\sf Hopf}(H)\to \Gal(H,H)$, $\alpha\to {^\alpha H}$ is a group homomorphism (see e.g.\ \cite{AV} for more details). For each $\alpha\in \Aut_{\sf Hopf}(H)$, we have a category of associated Yetter Drinfeld modules ${_{^{\alpha}H}\YD^{H}}(H,H)$ and take $\YD(H)=\bigsqcup_{\alpha\in \Aut_{\sf Hopf}(H)} {_{^\alpha H}\YD^{H}}(H,H)$. Then $\YD(H)$ is graded over $\Aut_{\sf Hopf}(H)$ and the monoidal structure is compatible with the grading since we know by Proposition~\ref{prop:MonoidalGenYD}
that 
$$\ot:{_{^\alpha H}\YD^{H}}(H,H)\times {_{^{\beta}H}\YD^{H}}(H,H)\to 
{_{{^\alpha H}\cotens^H {^\beta H}}\YD^{H}}(H,H)\cong {_{^{\alpha\beta}H}\YD^{H}}(H,H)$$
Furthermore, by Theorem~\ref{thm:generalbraiddual}(2) we have for any $\alpha\in \Aut_{\sf Hopf}(H)$ a functor $\phi(\alpha)=-\cotens^H {^{\alpha^{-1}}H}\cong -\cotens^H H^\alpha: {_{^\beta H}\YD^{H}}(H,H) \to {}_{_{\alpha\beta\alpha^{-1}}H}\YD^{H}(H,H)$. This functor leaves the underlying $k$-module unchanged and twists the right $H$-coaction with $\alpha$. It is clearly an equivalence of categories with inverse $\phi(\alpha^{-1})$, and it preserves the $k$-linear tensor product since $H^\alpha$ is a bi-Galois object. In this way we obtain a group morphism $\phi:\Aut_{\sf Hopf}(H)\to \Aut^{\ot}$.
The braidings as described in Theorem~\ref{thm:generalbraiddual} then turn $\YD(H)$ in a left $\Aut_{\sf Hopf}(H)$-crossed braided monoidal category. 

This example essentially goes back to \cite{panaite}. Although in that paper, the authors consider a double grading of $\Aut_{\sf Hopf}(H)$, on the category obtained as the disjoint union of all categories ${_H\YD^H}(\alpha,\beta)={_{^\alpha H^\beta}\YD^{H}}(H, H)$. Recall however from Lemma~\ref{le:alphabeta} that we have isomorphisms of categories 
$${_{^{\alpha\beta^{-1}} H}\YD^{H}}(H, H)\cong {_{^\alpha H^\beta}\YD^{H}}(H, H)\cong  {_{H^{\beta\alpha^{-1}}}\YD^{H}}(H, H),$$
induced by the isomorphisms of $H$-bicomodule algebras
$$\xymatrix{{^{\alpha\beta^{-1}} H} \ar[rr]_\cong^{\beta}  && {^\alpha H^\beta} \ar[rr]_\cong^{\alpha^{-1}} &&H^{\beta\alpha^{-1}} }$$
Therefore, the $k$-linear tensor product induces a monoidal structure of the form
\begin{eqnarray*}
{_{^\alpha H^\beta}\YD^H}(H,H)\times {_{^\gamma H^\delta}\YD^H}(H,H) \cong {_{^{\alpha\beta^{-1}}H}\YD^H}(H,H)\times {_{^{\gamma\delta^{-1}} H}\YD^H}(H,H) \\
&&\hspace{-5cm}\longrightarrow {_{^{\alpha\beta^{-1}\gamma\delta^{-1}}H}\YD^H}(H,H) \cong 
{_{^{\alpha\gamma}H^{\delta\gamma^{-1}\beta\gamma}}\YD^H}(H,H)
\end{eqnarray*}
which allows to endow the category 
$$\YD(H)=\bigsqcup_{(\alpha,\beta)\in \Aut_{\sf Hopf}(H)\times \Aut_{\sf Hopf}(H)} {_H\YD^H}(\alpha,\beta)$$
with a left $G$-crossed braided monoidal structure where $G=\Aut_{\sf Hopf}(H)\times \Aut_{\sf Hopf}(H)$ is the group with composition
$$(\alpha,\beta)*(\gamma,\delta)=(\alpha\gamma,\delta\gamma^{-1}\beta\gamma).$$
As could be seen from the above, this category however contains a lot of isomorphic copies of the same structures, so we believe the approach proposed at the beginning of this example is a suitable refinement of this structure.

Furthermore, in view of what comes next, it is useful to remark how the above construction can be dualized to obtain also a {\em right} group-crossed braided monoidal structure on the same category. Indeed, recall that for any $\alpha\in \Aut_{\sf Hopf}(H)$, we can also consider the associated bi-Galois co-object ${H_\alpha}$, whose coalgebra structure is the one from $H$, the left action is the regular one, while the right action is the regular action twisted by $\alpha$. Then we can write $\YD(H)=\bigsqcup_{\alpha\in \Aut_{\sf Hopf}(H)} {_H\YD^{_{\alpha}H}}(H,H)$ which is essentially the same $\Aut_{\sf Hopf}(H)$-graded monoidal category as before. However, now for any $\alpha\in \Aut_{\sf Hopf}(H)$, we consider a functor $\phi(\alpha)={_\alpha H}\ot_H-: {_H\YD^{_{\beta}H}}(H,H) \to {}_H\YD^{_{\alpha^{-1}\beta\alpha}H}(H,H)$. This functor leaves the underlying $k$-module unchanged and twists the left $H$-action by $\alpha$.  The braidings as described in Theorem~\ref{thm:generalbraid} then turn $\YD(H)$ into a right $\Aut_{\sf Hopf}(H)$-crossed braided monoidal category.

\end{example}

\begin{remarks}
Let us comment on some variations, refinements, and generalizations of group-crossed braided monoidal categories  that have already been defined over the last years. 

Firstly, already in \cite{TuraevArxiv}, a $k$-linear and rigid version of the notion is considered, where for rigidity, it is required that the dual of an object in $\Cc_g$ belongs to $\Cc_{g^{-1}}$. Although we will omit rigidity in what follows, it can still be incorporated into our generalizations.

Secondly, because of the decomposition of $\Cc$ into a disjoint union of full subcategories, the category $\Cc$ itself fails to have certain good properties. For example, because of the lack of (non-zero) morphisms between objects of different degree, the coproduct of two objects of different degree fails to exist. An elementary way of solving this issue is to consider instead of a disjoint union of subcategories $\Cc_g$,  a category $\Cc$ whose objects themselves are families $(X_g)_{g\in G}$, where $X_g\in \Cc_g$ for any $g\in G$, and morphisms are families of morphisms in each degree. In that way the coproduct of objects $X_g\in \Cc_g$ and $Y_h\in \Cc_h$ would be the family having $X_g$ in component $g$, $Y_h$ in component $h$, and zero everywhere else. This construction boils down to freely completing $\Cc$ with coproducts of objects of different degree.

More recently, S\"ozer and Virelizier \cite{Vir} provided a generalization of $G$-crossed braided monoidal categories, where the group $G$ is replaced with a crossed module. In this setting, a more refined way than a disjoint union is considered to treat the grading on the category. 

In what follows we aim to generalize the notion of crossed braided monoidal category in three ways. Firstly, we want to replace the group of strict monoidal automorphisms by the group of strong monoidal autoequivalences. In the setting of Example~\ref{ex:Panaite}, this means that we would replace Hopf algebra automorphisms by general bi-Galois (co-)objects.

Our second aim is to provide a bicategorical version of crossed braided monoidal categories. In \cite{Femic} an alternative bicategorical version is introduced, however our notion is quite different and more suited to establish the link with generalized Yetter-Drinfeld modules treated above.

Finally, because we aim to treat bi-Galois objects and co-objects on an equal footing, we will consider a double grading on our bicategory, one inducing a left crossing, the other a right crossing.
\end{remarks}

\subsection{Categories of bicomodule algebras and bimodule coalgebras, groupoids of bi-Galois objects and co-objects}

Recall from the beginning of section~\ref{braidingYD} that given a $(K,H)$-bimodule coalgebra $C$ and an $(L,K)$-bimodule coalgebra $D$, the balanced tensor product $D\ot_K C$ is an $(L,H)$-bimodule coalgebra, where $H$, $K$ and $L$ are any bialgebras over the same base ring $k$. Based on this construction, one can build a bicategory whose $0$-cells are bialgebras, $1$-cells are bimodule coalgebras, $2$-cells are morphisms of bimodule coalgebras and horizontal composition is given by the balanced tensor product. Furthermore, recall (see e.g.\ \cite{AV}) that when $C$ and $D$ are bi-Galois co-objects, then the tensor product $D\ot_K C$ is again a bi-Galois co-object, and furthermore all considered bialgebras $H$, $K$, $L$ are Hopf. 
By taking the truncation of this bicategory, we obtain the following definition.

\begin{definition}
The category $\BC$ is defined as the category whose objects are bialgebras, and morphisms are isomorphism classes of bimodule coalgebras, with composition given by the balanced tensor product.

The subcategory of $\BC$ consisting of those objects that are Hopf algebras and those morphisms that are bi-Galois co-objects, forms a groupoid that we denote as $\coGal$.
\end{definition}

The inverse of a morphism in $\coGal$ represented by the bi-Galois co-object $C$, is the isomorphism class represented by the bi-Galois co-object $\ol C$, as described in section \ref{se:Galois}, see also \cite{AV} for more details.

\begin{definition}
Dual to the above, we consider the category $\BA$ whose objects are bialgebras, and morphisms are isomorphism classes of $k$-flat bicomodule algebras, with composition given by the cotensor product (the $k$-flatness condition guarantees the associativity of the cotensor product). The subcategory of $\BA$ consisting of those objects that are Hopf algebras and those morphisms that are bi-Galois objects, forms a groupoid that we denote as $\Gal$.
\end{definition}

We can now combine the information of both categories and groupoids above. 

\begin{definition}
The category $\YD$ is defined as the category whose objects are bialgebras, and a morphism from a bialgebra $H$ to a bialgebra $K$ is an isomorphism class of couples $(A,C)$, where $A$ is a $k$-flat $(H,K)$-bicomodule algebra and $C$ is a $(K,H)$-bimodule coalgebra, such that $(H,K,A,C)$ is a YD datum. Composition is induced by the balanced cotensor product between bicomodule algebras and reversed balanced tensor product between bimodule coalgebras.

The subcategory of $\YD$ consisting of those objects that are Hopf algebras and those morphisms that are induced by a pair of a bi-Galois object and a bi-Galois co-object is denoted by $\biGal$.
\end{definition}

The above introduced categories and groupoids are related by obvious forgetful functors, represented by the diagonal arrows, and embedding functors, represented by the vertical arrows in the following diagram.
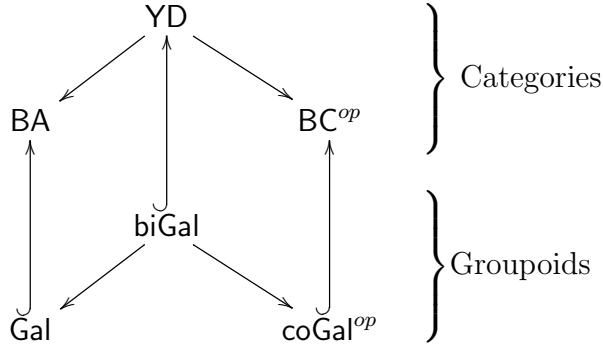
\[
\begin{tabular}{lr}
\begin{minipage}{6cm}
\xymatrix{
& \YD \ar[dl]  \ar[dr]\\
\BA  && \BC^{op} \\
& \biGal\ar@{^(->}[uu] \ar[dl] \ar[dr]\\
\Gal \ar@{^(->}[uu] && \coGal^{op} \ar@{^(->}[uu]
}\end{minipage} &
\begin{minipage}{6cm}
$\left.\begin{array}{c} \\ \\ \\ \\ \end{array}\right\}$ Categories \\
\vspace{.4cm} \\
$\left.\begin{array}{c} \\ \\ \\ \\ \end{array}\right\}$Groupoids
\end{minipage}
\end{tabular}
\]

\subsection{Bicategories graded by categories}
Let is introduce the following definition, as a 2-dimensional version of Turaev's group-graded categories.

\begin{definition}\label{def:gradedbicat}
Let $\sf C$ be a category and $\Bb$ be a bicategory. We say that $\Bb$ is $\sf C$-graded if there exists a $2$-functor $\partial:\Bb\to \sf C$, where $\sf C$ is considered as a discrete bicategory, i.e.\ with only identity $2$-cells.
\end{definition}

Let us observe that the existence of $\partial$ means that for any pair of zero-cells $A,B$ in $\Bb$, we can write their Hom-category as the following disjoint union: 
$$\Bb(A,B) = \bigsqcup_{g\in {\sf C}(\partial A,\partial B)} \Bb_g(A,B),$$
where $\Bb_g(A,B)$ is the inverse image of $g$ under the $2$-functor $\partial$. Furthermore, the horizontal composition in $\Bb$, which we denote as $\ot:\Bb(A,B)\times \Bb(B,C)\to \Bb(A,C)$, is compatible with the composition in $\sf C$ as follows
$$\ot: \Bb_g(A,B)\times \Bb_h(B,C) \to \Bb_{gh},$$
or otherwise stated
$$\partial(X\ot Y)=\partial X\partial Y.$$
Remark that we denoted the horizontal composition in $\Bb$ as well as the composition in the grading category $\sf C$ in the opposite way as compared to what is usually common, to make it more compatible with the compositions given by tensor products in the next examples. When $\Bb$ has one zero-cell (hence is a monoidal category) and $\sf C$ is a one-object category (hence a monoid), we recover the notion of a graded monoidal category in the sense of Turaev as described the previous subsection.

\begin{examples}\label{ex:*examples}
We denote by ${_*\Mod}$ the bicategory defined as follows:
\begin{itemize}
\item zero-cells are bialgebras;
\item a one-cell from $H$ to $K$ is a pair $(A,M)$, where $A$ is a $k$-flat $(H,K)$-bicomodule algebra and $M$ is a left $A$-module. Horizontal composition of one cells $(A,M):H\to K$ and $(B,N):K\to L$ is given by $(A\cotens^K B, M\ot N):H\to L$.
\item a two-cells from $(A,M)$ to $(B,N)$ is a pair $(\alpha,f)$ where $\alpha:A\to B$ is an $(H,K)$-bicomodule algebra isomorphism and $f:M\to N$ is $A$-linear map (where the $A$-module structure on $N$ obtained from restriction of scalars by $\alpha$).
\end{itemize}
Then ${_*\Mod}$ is a $\BA$-graded bicategory, where $\partial: {_*\Mod}\to \BA$ sends a pair $(A,M)$ to the isomorphism class of $A$. Otherwise stated, Hom-category in ${_*\Mod}$ between any pair of bialgebras $H,K$ decomposes as follows
$${_*\Mod}(H,K)= \bigsqcup_{A\in \BA(H,K)} {_A\Mod},$$
where we recall that $\BA(H,K)$ is the category of isomorphism classes of $(H,K)$-bicomodule algebras.

Dually, one defines $\Mod^*$ as the $\BC^{op}$-graded bicategory whose zero-cells are bialgebras and one-cells are comodules over bimodule coalgebras. Combining both, we obtain ${_*\YD^*}$ as the $\YD$-graded bicategory whose zero-cells are bialgebras and one-cells are generalized Yetter Drinfeld-modules over arbitrary YD data. This means that for any pair of bialgebras $(H,K)$, we have
$${_*\YD^*}(H,K) = \bigsqcup_{(A,C)~|~(H,K,A,C)\in \YD} {_A\YD^C}(H,K)$$
and horizontal composition in this bicategory is given by the $k$-linear tensor product, which is a reinterpretation of Proposition~\ref{prop:MonoidalGenYD}. 
\end{examples}

The above examples can then be summarized and related in the following diagram, where the upper diagonal arrows are the obvious forgetful $2$-functors, and the lower diagonal arrows are the obvious forgetful functor we already encountered in the previous paragraph.
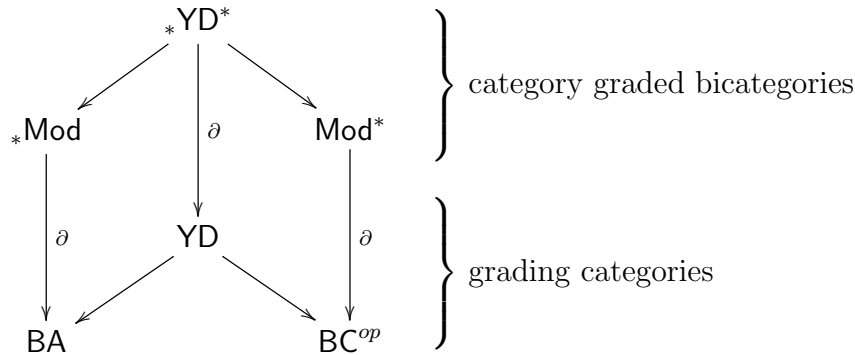
\[
\begin{tabular}{lr}
\begin{minipage}{6cm}
\xymatrix{
& {_*\YD^*} \ar[dl] \ar[dd]^\partial \ar[dr]\\
{_*\Mod} \ar[dd]^\partial && \Mod^* \ar[dd]^\partial\\
& \YD \ar[dl] \ar[dr]\\
\BA && \BC^{op}
}\end{minipage} &
\begin{minipage}{6cm}
$\left.\begin{array}{c} \\ \\ \\ \\ \end{array}\right\}$ category graded bicategories \\
\vspace{.4cm} \\
$\left.\begin{array}{c} \\ \\ \\ \\ \end{array}\right\}$ grading categories
\end{minipage}
\end{tabular}
\]

\begin{examples}\label{circledexamples}
We can also consider variations of the above bicategories, whose gradings land in the groupoids of bi-Galois (co)objects. More precisely, we denote by $${_\circledast\Mod}$$ the full sub-bicategory of ${_*\Mod}$ consisting of the same zero-cells and those one-cells whose grading lands in $\Gal$. This means that a $1$-cell from $H$ to $K$ in ${_\circledast\Mod}$ is a pair $(A,M)$ where $A$ is an $(H,K)$ bi-Galois object and $M$ is a left $A$-module.

Similarly, we consider $\Mod^\circledast$ as the full sub-bicategory of $\Mod^*$ whose one-cells are graded by bi-Galois co-objects. In the same way we can consider $${_\circledast\YD^\circledast}$$ as the full sub-category of ${_*\YD^*}$ described in Examples~\ref{ex:*examples}, which is graded over $\biGal$. Finally, we also consider the categories ${_\circledast\YD^1}$ and ${_1\YD^\circledast}$, which are further restrictions of ${_\circledast\YD^\circledast}$ where we only consider those Yetter-Drinfeld data where the bimodule coalgebra, respectively the bicomodule algebra, is the trivial one. The latter means that for two different Hopf algebras $H,K$, the Hom category ${_\circledast\YD^1}(H,K)$ contains no (non-zero) objects and
$${_\circledast\YD^1}(H,H)=\bigsqcup_{A\in \Gal(H,H)}{_A\YD^H}(H,H)$$
is a $\Gal(H,H)$-graded monoidal category. 
We can again summarize the relations between these structures in the following diagram
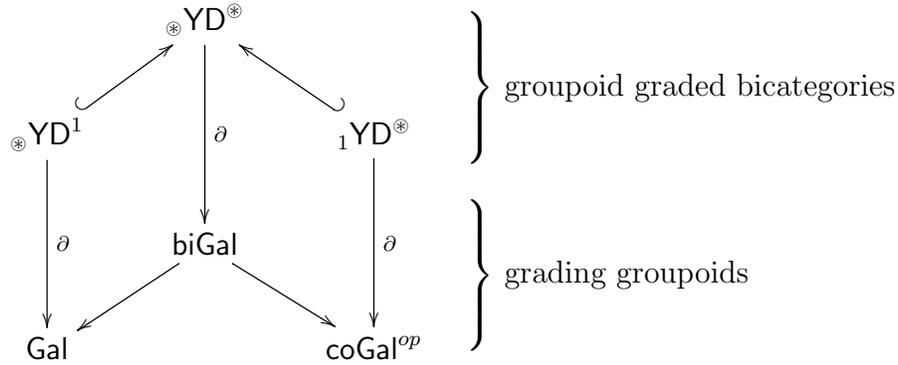
\[
\begin{tabular}{lr}
\begin{minipage}{6cm}
\xymatrix{
& {_\circledast\YD^\circledast}  \ar[dd]^\partial \\
{_\circledast\YD^1} \ar[dd]^\partial \ar@{^(->}[ur] && {_1\YD^\circledast} \ar@{_(->}[ul] \ar[dd]^\partial\\
& \biGal \ar[dl] \ar[dr]\\
\Gal && \coGal^{op}
}\end{minipage} &
\begin{minipage}{6cm}
$\left.\begin{array}{c} \\ \\ \\ \\ \end{array}\right\}$ groupoid graded bicategories \\
\vspace{.4cm} \\
$\left.\begin{array}{c} \\ \\ \\ \\ \end{array}\right\}$ grading groupoids
\end{minipage}
\end{tabular}
\]
\end{examples}

\subsection{Groupoid-crossed bicategories}

Our next aim is to provide a 2-dimensional version of group-crossed braided monoidal categories. Clearly, in this setting we will have to consider gradings over groupoids rather than mere categories. To formulate this definition, let us first recall some definitions and introduce notation.

\begin{definition}
Let $\Cc$ and $\Dd$ be two categories. An {\em opfunctor} (also sometimes called a {\em cofunctor} or a {\em retrofunctor} in literature) $F:\Cc\to\Dd$, assigns to each object $D\in\Dd^0$ an object $F^0D\in\Cc^0$ and to each morphism $f\in\Cc(FD,FD')$ a morphism $Ff\in \Dd(D,D')$ such that $F(\Id_{FD})=\Id_D$ and $F(g\circ f)=F(g)\circ F(f)$ for all $D,D',D''\in\Dd^0$ and all $f\in\Cc(FD,FD')$, $g\in \Cc(FD',FD'')$. In particular, if $\sf G$ and $\sf H$ are groupoids, then an opfunctor ${\sf G}\to{\sf H}$ is also called a {\em groupoid opmorphism}.
\end{definition}

\begin{notation}
Let $\Bb$ be a bicategory. Then for any zero-cell $A$ of $\Bb$, the endohom category $\Bb(A,A)$ is monoidal by means of the horizontal composition of $\Bb$. We denote by $\Aut^\ot(\Bb)$ the groupoid whose objects are the zero-cells of $\Bb$ and a morphism from $A$ to $B$ is a strong monoidal equivalence $F:\Bb(A,A)\to \Bb(B,B)$, up to monoidal natural isomorphism and composition in $\Aut^\ot(\Bb)$ is induced by the usual composition of functors (hence opposed to the convention for composition in grading categories used in the previous subsection).
This notation is motivated by the fact that for any zero cell $A\in \Bb$, we have $\Aut^\ot(\Bb)(A,A)=\Aut^\ot(\Bb(A,A))$, where the latter denotes the strong monoidal autoequivalences of the monoidal category $\Bb(A,A)$. More generally, for any pair of zero cells $A,B$ in $\Bb$, we have $\Aut^\ot(\Bb)(A,B)={\sf Iso}^\ot(\Bb(A,A),\Bb(B,B))$.
\end{notation}

\begin{definition}
Let $\sf G$ be a groupoid and $\Bb$ a bicategory, whose horizontal composition we denote by $\ot$. We say that $\Bb$ is a {\em right $\sf G$-crossed braided bicategory} if it is endowed with the following structure.
\begin{itemize}
\item $\Bb$ is $\sf G$ graded via the $2$-functor $\partial:\Bb\to {\sf G}$ (see Definition~\ref{def:gradedbicat}).
\item There is an opmorphism of groupoids $\phi:{\sf G}^{op}\to \Aut^\ot(\Bb)$, where $\phi^0=\partial$. Explicitly, this means that for any pair of zero cells $A,B$ in $\Bb$, and any $g\in {\sf G}(\partial B,\partial A)$, there is a strong monoidal functor
$$\phi(g)=(-)^g:\Bb(A,A)\to \Bb(B,B),$$
such that $(X^g)^h \cong X^{gh}$ for composable morphisms $g,h\in {\sf G}$. The strong monoidality of $\phi(g)$ is expressed by natural isomorphisms
$$(X\ot Y)^g\cong X^g\ot Y^g,$$
satisfying suitable compatibility conditions.
\item The monoidal functors $\phi(g)$ respect the gradings in the sense that ${(-)}^g:\Bb_h(A,A)\to \Bb_{g^{-1}hg}(B,B)$ for all $h\in{\sf G}(\partial A,\partial A)$, or otherwise stated
$$\partial(X^g) = g^{-1}\partial X g,$$
for all $X\in \Bb(A,A)$.
\item There braiding of the following kind:
For $X\in \Bb(A,A)$ and $Y\in \Bb(A,B)$, we have
$$\beta_{X,Y} : X\ot Y \to Y \ot X^{\partial Y}$$
which is a natural isomorphism in $\Bb_{\partial X\partial Y}(A,B)$, satisfying the following heptagon\footnote{The left diagram has to be compared to the heptagon condition from \eqref{prebraidedact1}. As we omitted the associativity constraints, 3 additional arrows have been left out of the diagram.} and hexagon\footnote{The right diagram has to be compared to the two hexagon conditions of \eqref{prebraidedact2}.} conditions:
\begin{equation}\label{hexagonsbicat}
\xymatrix{
X\ot Y \ot Z \ar[rr]^-{\Id_X\ot \beta_{Y,Z}} \ar[d]_-{\beta_{(X\ot Y),Z}} && X\ot Z\ot Y^{\partial Z} \ar[d]^-{\beta_{X,Z}\ot \Id_{Y^{\partial Z}}}\\
Z\ot (X\ot Y)^{\partial Z} \ar[rr]^\cong && Z\ot X^{\partial Z}\ot Y^{\partial Z}
}\quad
\xymatrix{
X\ot Y \ot Z \ar[rr]^-{\beta_{X,Y}\ot \Id_Z} \ar[d]_-{\beta_{X, (Y\ot Z)}} && Y\ot X^{\partial Y}\ot Z \ar[d]^-{\Id_Y \ot \beta_{X^{\partial Y},Z}}\\
Y\ot Z\ot X^{\partial(Y\ot Z)} \ar@{=}[rr] && Y\ot Z\ot X^{\partial Y\partial Z}
}
\end{equation}
\end{itemize}
\end{definition}

\begin{remarks}
Remark that the term ``right'' comes from the fact that {\sf G} acts via $\phi$ from the right on $\Bb$ by monoidal equivalences between its endohom categories. Similarly, one can consider a {\em left} $\sf G$-crossed braided bicategory. In this case, $\Bb$ is again $\sf G$-graded via $\partial :\Bb\to \sf G$, but we have a groupoid opmorphism $\phi:{\sf G}\to \Aut^\ot(\Bb),\ g\mapsto \phi(g)={^g(-)}$ such that $\partial({^g X})=g\partial X g^{-1}$ and the braidings are of the type $\beta_{X,Y} : X\ot Y \to {^{\partial X}Y} \ot X$.
\end{remarks}

\begin{proposition}
The bicategory ${_\circledast\YD^1}$ from Examples~\ref{circledexamples} is a left $\Gal$-crossed braided monoidal category. The bicagegory ${_1\YD^\circledast}$ is a right $\coGal^{op}$-crossed braided monoidal category.
\end{proposition}

\begin{proof}
In Examples~\ref{ex:*examples} and Examples~\ref{circledexamples} we described the $\Gal$-grading of ${_\circledast\YD^1}$ and the $\coGal^{op}$-grading of ${_1\YD^\circledast}$. Furthermore, by Theorem~\ref{thm:generalbraiddual}(2) and Corollary~\ref{GaloisCoobjectlifts} we know that we have well-defined opmorphisms\footnote{Since in this example, both groupoids have the same set of objects, this is in fact also a morphism of groupoids.} of groupoids
\begin{eqnarray*}
\Gal \to \Aut^\ot({_\circledast\YD^1}),\ && A \mapsto -\cotens^H\ol A : {_B\YD^H}(H,H)\to {_{A\cotens^HB\cotens^H\ol A}\YD^{H}}(H,H)\\
\coGal \to \Aut^\ot({_1\YD^\circledast}),\ && C \mapsto C\ot_H- : {_H\YD^D}(H,H)\to {_H\YD^{C\ot_H D\ot_H\ol C}}(H,H)
\end{eqnarray*}
where $A,B$ are an $(H,H)$-bi-Galois objects and $\ol A$ is the inverse of $A$ in $\Gal$; and $C,D$  are $(H,H)$-bi-Galois co-objects and $\ol C$ is the inverse of $C$ in $\coGal$.
Finally, the crossed-braided structure follows from the braiding described in Theorem~\ref{thm:generalbraiddual}(4) and Theorem~\ref{thm:generalbraid}, again specialized to the case where both appearing bialgebras coincide.
\end{proof}

As remarked before, rather than being a true bicategory, ${_\circledast\YD^1}$ is merely a disjoint union of monoidal categories of the form ${_\circledast\YD^1}(H,H)$, since there are no one-cells in ${_\circledast\YD^1}(H,K)$ when $H$ and $K$ are different. In fact, each of these endohom categories is a $\Gal(H,H)$-crossed braided monoidal category, generalizing Example~\ref{ex:Panaite}. This holds in general, as the next result states. In order to describe properly the structure of the more interesting bicategory ${_\circledast\YD^\circledast}$, we will combine the two crossed structures over $\Gal$ and $\coGal$ in the next subsection.

\begin{proposition}
Let $\sf G$ be a groupoid and $\Bb$ a right $\sf G$-crossed braided bicategory. For any $0$-cell $A\in \Bb$, the endohom category $\Bb(A,A)$ is a ${\sf G}(\partial A,\partial B)$-crossed braided monoidal category.
\end{proposition}

\subsection{Double Groupoid-crossed bicategories}

\begin{definition}\label{def:groupoidcrossedbicat}
Let $\Bb$ be a bicategory and ${\sf H}$ and ${\sf G}$ be two groupoids. We say that $\Bb$ is $({\sf H},{\sf G})$-double crossed braided if and only if $\Bb$ is equiped with the following structure.
\begin{itemize}
\item $\Bb$ is ${\sf H}\times{\sf G}$-graded, via the $2$-functors $\partial':\Bb\to {\sf H}$ and $\partial:\Bb\to {\sf G}$. Explicitly, we write for each pair of zero-cells $A,B$ in $\Bb$ the gradings as 
$$\Bb(A,B)
=\bigsqcup_{{h\in {\sf H}(\partial' A,\partial' B)}} {_h\Bb}(A,B)
=\bigsqcup_{{g\in {\sf G}(\partial A,\partial B)}} {\Bb_g}(A,B)
=\bigsqcup_{^{h\in {\sf H}(\partial' A,\partial' B)}_{g\in {\sf G}(\partial A,\partial B)}} {_h\Bb_g}(A,B).$$
\item There are opmorphisms of groupoids $\phi':{\sf H}\to \Aut^\ot(\Bb_{\Id_{\sf G}})$ and $\phi:{\sf G}^{op}\to \Aut^\ot(_{\Id_{\sf H}}\Bb)$, with ${\phi'}^0=\partial'$ and $\phi^0=\partial$. Here $\Bb_{\Id_{\sf G}}$ denotes the full sub-bicategory of $\Bb$ consisting of those components with a trivial $\sf G$-grading and ${_{\Id_{\sf H}}\Bb}$ denotes the full sub-bicategory of $\Bb$ consisting of those components with a trivial $\sf H$-grading. More explicitly, the above means that for any pair of zero-cells $A,B$ in $\Bb$ and for any morphisms $h\in {\sf H}(\partial'A, \partial'B)$ 
and $g\in{\sf G}(\partial A,\partial B)$ 
we have strong monoidal functors
$$\phi'(h)={^h(-)}:\Bb_{\partial B}(B,B)\to \Bb_{\partial A}(A,A) , \qquad \phi(g)=(-)^g:{_{\partial' A}\Bb}(A,A)\to {_{\partial' B}\Bb}(B,B).$$
\item The opmorphisms $\phi'$ and $\phi$ are compatible with the grading in the sense that (using same notation as in the previous item) 
$${^h(-)}:{_{h'}\Bb}_{\partial B}(B,B)\to {_{hh'h^{-1}}\Bb}_{\partial A}(A,A), \qquad (-)^g:{_{\partial' A}\Bb_{g'}}(A,A)\to {_{\partial' B}\Bb_{g^{-1}g'g}}(B,B)$$
for all $h'\in {\sf H}(\partial' B,\partial' B)$ and $g'\in{\sf G}(\partial A,\partial A)$. In other words
$$\partial'({^h X})=h\partial' X h^{-1}, \qquad \partial(Y^g)=g^{-1}\partial Yg,$$
for all $X\in \Bb_{\partial B}(B,B)$ and $Y\in {_{\partial' A}\Bb}(A,A)$. 
\item There are braidings of the following type: for $X\in{_{\partial' A}\Bb_{g'}}(A,A)$, $Y\in{_h\Bb_g}(A,B)$ and $Z\in {_{h'}\Bb_{\partial B}}(B,B)$, we have
$$\beta'_{Y,Z}:Y\ot Z\to {^{\partial' Y}Z}\ot Y\in {_{hh'}\Bb_{g}}(A,B),\qquad \beta_{X,Y}:X\ot Y\to Y\ot X^{\partial Y} \in{_h\Bb_{g'g}}(A,B)$$
subjected to the hexagon conditions \eqref{hexagonsbicat} and their left counterparts for $\beta'$.
\end{itemize}
\end{definition}

\begin{theorem}
The bicategory ${_\circledast\YD^\circledast}$ described in Example~\ref{circledexamples} is a $(\Gal,\coGal^{op})$ double crossed braided bicategory.
\end{theorem}

\begin{proof}
The gradings of ${_\circledast\YD^\circledast}$ by $\Gal$ and $\coGal^{op}$ have been described in Example~\ref{ex:*examples} and Example~\ref{circledexamples}. From Theorem~\ref{thm:generalbraiddual}(2) and Corollary~\ref{GaloisCoobjectlifts} we obtain the groupoid opmorphisms $\phi'$ and $\phi$, given by
\begin{eqnarray*}
\Gal \to \Aut^\ot({_\circledast\YD^1}),\ && A \mapsto -\cotens^K\ol A : {_B\YD^K}(K,K)\to {_{A\cotens^KB\cotens^K\ol A}\YD^{H}}(H,H)\\
\coGal \to \Aut^\ot({_1\YD^\circledast}),\ && C \mapsto C\ot_H- : {_H\YD^D}(H,H)\to {_K\YD^{C\ot_H D\ot_H\ol C}}(K,K)
\end{eqnarray*}
where $A$ is an $(H,K)$-bi-Galois object with inverse $\ol A$ in $\Gal$ and $C$ is a $(K,H)$-bi-Galois co-object with inverse $\ol C$ in $\coGal$.
The crossed-braided structure is obtained from the braidings described in Theorem~\ref{thm:generalbraiddual}(4) and Theorem~\ref{thm:generalbraid}. 
\end{proof}

In order to describe the structure of the full category ${_*\YD^*}$ graded over {\em all} YD-data, we need one last generalization of Definition~\ref{def:groupoidcrossedbicat}.

\begin{definition}
Let $\sf D$ and $\sf C$ be two categories and $\sf H$, $\sf G$ two groupoids together with functors  $\iota':{\sf H}\to {\sf D}$ and $\iota:{\sf G}\to {\sf C}$. We say that a bicategory $\Bb$ is {\em $({\sf D},{\sf C})$-double graded $({\sf H},{\sf G})$-double crossed braided} if and only if:
\begin{itemize}
\item $\Bb$ is $\sf D\times\sf C$-graded via $\partial':\Bb\to \sf D$ and $\partial:\Bb\to \sf C$.
\item There are opmorphisms of groupoids $\phi':{\sf H}\to \Aut^\ot(\Bb_{1_{\sf G}})$ and $\phi:{\sf G}^{op}\to \Aut^\ot(_{1_{\sf H}}\Bb)$, with $\iota'\circ {\phi'}^0=\partial'$ and $\iota\circ \phi^0=\partial$.
\item The opmorphisms $\phi'$ and $\phi$ are compatible with the grading in the sense that
$$\partial'({^h X})=\iota'(h)\partial' X\iota'(h^{-1}), \qquad \partial(Y^g)=\iota(g^{-1})\partial Y\iota(g),$$
for all morphisms $h\in\sf H$, $g\in \sf G$ and one-cells $X$ and $Y$ in $\Bb$, such that sources and targets of $\iota'(h)$ and $\iota(g)$ match with the gradings on $X$ and $Y$ for the above formula to make sense.
\item There are braidings of the following type: for $X\in{_{\partial' A}\Bb_{d'}}(A,A)$, $Y\in{_c\Bb_{\iota(g)}}(A,B)$, $Y'\in{_{\iota'(h)}\Bb_d}(A,B)$ and $Z\in {_{c'}\Bb_{\partial B}}(B,B)$, we have
$$\beta'_{Y,Z}:Y\ot Z\to {^{\partial' Y}Z}\ot Y\in {_{\iota'(h)c'}\Bb_{d}}(A,B),\qquad \beta_{X,Y}:X\ot Y\to Y\ot X^{\partial Y} \in{_c\Bb_{d'\iota(g)}}(A,B).$$
\end{itemize}
\end{definition}

The braidings given in Theorem~\ref{thm:generalbraiddual}(4) and Theorem~\ref{thm:generalbraid} can now exactly be reformulated as the following result.

\begin{theorem}
The bicategory ${_*\YD^*}$ is a $(\BA,\BC^{op})$-double graded $(\Gal,\coGal^{op})$-double crossed braided bicategory.
\end{theorem}

\begin{remarks}
One can also consider {\em rigid} groupoid-crossed braided bicategories, where the grading of the dual of object $X$ is given exactly by the inverse in the groupoid of the grading of $X$. By considering the subbicategory of ${_\circledast\YD^\circledast}$ whose $1$-cells are finite-dimensional generalized Yetter-Drinfeld modules, one could also show this is rigid. 

As mentioned earlier, one can refine the notion of groupoid gradings as described above by disjoint unions, by freely completing those with coproducts of objects of different degree. Objects and morphisms in such a completion are then described as families of objects of homogeneous degree. 

Finally, it would be interesting to explore the use of groupoid-crossed braided bicategories, and in particular the bicategory of generalized Yetter-Drinfeld modules in Homotopy Quantum Field Theories. For example the crossed structures appearing in \cite{Paul} and \cite{Vir} share interesting similarities with our constructions.
This should be the subject of future investigations.
\end{remarks}


\begin{thebibliography}{1}

\bibitem{AV}
R.\ Aziz and J.\ Vercruysse, \textit{Bi-Galois co-objects and Hopf categories}, in preparation.


\bibitem{brzezinski} T. Brzezi\'nski, \textit{Hopf-cyclic homology with contramodule coefficients}, in \textit{Quantum groups and noncommutative spaces}, 1--8, 2008.


\bibitem{CDV}
S.\ Caenepeel, E.\ De Groot, and J.\ Vercruysse, \textit{Galois theory for Comatrix Corings: Descent theory, Morita theory, Frobenius and separability properties, Trans. Amer. Math. Soc.} \textbf{359} (2007), 185--226.

\bibitem{BookStef} S. Caenepeel, G. Militaru, and S. Zhu, \textit{Frobenius and separable functors for generalized module categories and nonlinear equations}, Springer, 2004.

\bibitem{canaepeel} S. Caenepeel, F. Van Oystaeyen, and Y. Zhang, \textit{The Brauer group of Yetter-Drinfeld module algebras, Trans. Amer. Math. Soc.} \textbf{349} (1997), 3737--3771.

\bibitem{Drinfeld}
V.\ Drinfeld, \textit{Quantum groups}, in A. Gleason (ed.), \textit{Proc. Int. Congr. Math. 1986}, Vol. 1, 798--820 1987



\bibitem{Femic}
B.\ Femi\'c,
\textit{Turaev bicategories and generalized Yetter-Drinfeld modules in 2-categories, Israel J. Math.} \textbf{241} (2021), 395--432.

\bibitem{gelaki} S. Gelaki, D. Naidu, and D. Nikshych, \textit{Centers of graded fusion categories, Algebra \& Number Theory}, \textbf{3} (2009), 959--990.

\bibitem{Paul}
P.\ Gro\ss kopf, \textit{2D HQFTs and Frobenius $(\Gg,\Vv)$-categories}, preprint, 2025, arXiv:2501.10113.

\bibitem{hajac} P. M. Hajac, M. Khalkhali, B. Rangipour, and Y. Sommerh\"auser, \textit{Stable Anti-Yetter-Drinfeld Modules, C. R. Acad. Sci. Paris}, Ser. I \textbf{336} (2003).

\bibitem{hassanzadeh} M. Hassanzadeh, M. Khalkhali, and I. Shapiro, \textit{Monoidal categories, 2-traces, and cyclic cohomology, Canadian Mathematical Bulletin}, \textbf{62} (2019), 293--312.

\bibitem{Laugwitz}
R.\ Laugwitz, Robert, C.\ Walton, M.\ Yakimov,
\textit{Reflective centers of module categories and quantum $K$-matrices, Forum Math. Sigma} \textbf{13} (2025), Paper No. e95.

\bibitem{maclane} S. Mac Lane. \textit{Categories for the working mathematicians}, Springer-Verlag, 1971.

\bibitem{ma:center} S. Majid,  \textit{Representations, duals and quantum doubles of monoidal categories} in \textit{Proceedings of the Winter School ``Geometry and Physics"}, 197 -- 206. Circolo Matematico di Palermo, 1991.

\bibitem{ma:book} S. Majid. \textit{Foundations of quantum group theory}, Cambridge University Press, 2000.

\bibitem{mesablishvili}
B. Mesablishvili,
\textit{Monads of effective descent type and comonadicity,
Theory Appl. Categ.} \textbf{16} (2006), 1--45.

\bibitem{panaite} F. Panaite and M. D. Staic, \textit{Generalized (anti) Yetter-Drinfeld modules as components of a braided T-category, Israel Journal of Mathematics}, \textbf{158} (2007), 349--365.

\bibitem{Schauenburg:biGalois}
P. Schauenburg, \textit{Hopf-Galois and bi-Galois extensions, Fields Inst. Commun.}, \textbf{43}, American Mathematical Society, Providence, RI, 2004, 469--515.

\bibitem{Vir} K.\ S\"ozer and A.\ Virelizier, \textit{Monoidal categories graded by crossed modules and 3-dimensional HQFTs, Advances in Math.}, \textbf{428} (2023) , 109155.

\bibitem{TuraevArxiv}
V.\ G.\ Turaev, \textit{Homotopy field theory in dimension 3 and crossed group-categories}, preprint, 2000,  arXiv:math.GT/0005291.

\bibitem{TuraevBook}
V.\ G.\ Turaev, 
\textit{Homotopy quantum field theory}, with appendices by Michael M\"uger and Alexis Virelizier, \textit{EMS Tracts in Mathematics} \textbf{10}, European Mathematical Society, Z\"urich, 2010.
\end{thebibliography}
\end{document}